\documentclass[12pt, reqno]{amsart}

\usepackage{mathtools}
\mathtoolsset{showonlyrefs}
\numberwithin{equation}{section}

\usepackage{mathrsfs}

\usepackage{hyperref}
\usepackage[dvipdfmx]{graphicx}
\usepackage{amssymb}
\usepackage{amsthm}
\usepackage{amscd}
\usepackage{amsmath}
\usepackage{epic,eepic}
\usepackage{mathrsfs}
\usepackage{latexsym}
\usepackage{pifont}
\usepackage[all]{xy}
\usepackage {cancel}
\usepackage{tikz}
\usepackage{mathrsfs}
\DeclareMathAlphabet{\mathpzc}{OT1}{pzc}{m}{it}

\theoremstyle{plain}
\newtheorem{lemma}{Lemma}[section]
\newtheorem{prop}[lemma]{Proposition}
\newtheorem{thm}[lemma]{Theorem}
\newtheorem{cor}[lemma]{Corollary}
\newtheorem{intthm}{Theorem}

\theoremstyle{definition}
\newtheorem{definition}[lemma]{math}
\newtheorem{rem}[lemma]{Remark}
\newtheorem{defi}[lemma]{Definition}
\newtheorem{exa}[lemma]{Example}

\newcommand{\bde}{\begin{defi}}
\newcommand{\ede}{\end{defi}\vspace{1mm}}
\newcommand{\ble}{\begin{lemma}}
\newcommand{\ele}{\end{lemma}}
\newcommand{\bpr}{\begin{prop}}
\newcommand{\epr}{\end{prop}}
\newcommand{\bt}{\begin{thm}}
\newcommand{\et}{\end{thm}}
\newcommand{\bco}{\begin{cor}}
\newcommand{\eco}{\end{cor}}
\newcommand{\bre}{\begin{rem}}
\newcommand{\ere}{\end{rem}}
\newcommand{\bex}{\begin{exa}}
\newcommand{\eex}{\end{exa}}
\newcommand{\bpf}{\begin{proof}}
\newcommand{\epf}{\end{proof}}

\newcommand{\mcB}{\mathcal{B}}
\newcommand{\mcC}{\mathcal{C}}
\newcommand{\mcD}{\mathcal{D}}
\newcommand{\mcE}{\mathcal{E}}
\newcommand{\mcF}{\mathcal{F}}
\newcommand{\mcG}{\mathcal{G}}
\newcommand{\mcH}{\mathcal{H}}

\newcommand{\mcL}{\mathcal{L}}
\newcommand{\mcM}{\mathcal{M}}
\newcommand{\mcN}{\mathcal{N}}
\newcommand{\mcO}{\mathcal{O}}
\newcommand{\mcP}{\mathcal{P}}
\newcommand{\mcQ}{\mathcal{Q}}
\newcommand{\mcR}{\mathcal{R}}
\newcommand{\mcS}{\mathcal{S}}
\newcommand{\mcT}{\mathcal{T}}
\newcommand{\mcU}{\mathcal{U}}

\newcommand{\mbF}{\mathbb{F}}

\newcommand{\mbQ}{\mathbb{Q}}
\newcommand{\mbR}{\mathbb{R}}

\newcommand{\mbZ}{\mathbb{Z}}

\newcommand{\mfe}{\mathfrak{e}}
\newcommand{\mff}{\mathfrak{f}}
\newcommand{\mfg}{\mathfrak{g}}
\newcommand{\mfh}{\mathfrak{h}}

\newcommand{\mfm}{\mathfrak{m}}

\newcommand{\msE}{\mathscr{E}}
\newcommand{\msF}{\mathscr{F}}
\newcommand{\msG}{\mathscr{G}}
\newcommand{\msH}{\mathscr{H}}

\newcommand{\msO}{\mathscr{O}}

\newcommand{\msU}{\mathscr{U}}

\newcommand{\msX}{\mathscr{X}}

\newcommand{\LSP}{\vspace{5mm}}
\newcommand{\mr}{\mathrm}
\newcommand{\N}{m}
\newcommand{\LL}{\ell}

\newcommand{\NN}{N}
\newcommand{\DMO}{\nabla}
\newcommand{\EX}{d}
\newcommand{\M}{m}

\usepackage{bm}
\usepackage {cancel}

\begin{document}

\title[Frobenius pull-back of parabolic bundles and dormant opers]{Frobenius pull-back of parabolic bundles \\ and dormant opers}
\author{Yasuhiro Wakabayashi}
\markboth{}{}
\maketitle
\footnotetext{Y. Wakabayashi: 
Graduate School of Information Science and Technology, The University of Osaka, Suita, Osaka 565-0871, Japan;}
\footnotetext{e-mail: {\tt wakabayashi@ist.osaka-u.ac.jp};}
\begin{abstract}
We study parabolic bundles on an algebraic curve in positive characteristic. Our motivation is to properly formulate Frobenius pull-backs of parabolic bundles in a way that extends various previous facts and arguments for the usual non-parabolic Frobenius pull-backs. After defining that operation, we generalize a classical result by Cartier concerning Frobenius descent, that is, we establish a bijective correspondence (including the version using higher-level $\mathcal{D}$-modules) between parabolic flat bundles with vanishing $p$-curvature on a pointed curve and parabolic bundles on its Frobenius twist. This correspondence gives a description of maximally Frobenius-destabilized parabolic bundles in terms of dormant opers admitting logarithmic poles. As an application of that description together with a previous result in the enumerative geometry of dormant opers, we obtain an explicit formula for computing the number of such parabolic bundles of rank $2$ under certain assumptions.

\end{abstract}
\tableofcontents

\section{Introduction}
Let 
$(g, r)$ be a pair of nonnegative integers with $2g-2+r >0$ and
$\msX := (X, \{ \sigma_i \}_{i=1}^r)$  an $r$-pointed smooth proper  curve of genus $g$ over an algebraically closed field $k$.
A {\it parabolic (vector) bundle} is a vector bundle equipped with weighted  flags  on its fibers over the  marked points $\sigma_i$.
To be more precise, 
 a parabolic bundle on $\msX$ means  a collection 
\begin{align} \label{EQ646}
\msF := (\mcF, \vec{\mff}, \vec{\alpha})
\end{align}
(cf. Definition \ref{Def4453})
 consisting of a vector bundle (i.e., a locally free coherent sheaf) $\mcF$ on $X$,  a set of parabolic weights $\vec{\alpha}$, and 
 an $r$-tuple $(\mff_1, \cdots, \mff_r)$, where $\mff_i$ ($i=1, \cdots, r$)
  denotes  a sequence of $k$-linear surjections
\begin{align} \label{EQ647}
\sigma_i^*\mcF =: \mcF_i^{[\N_i]} \twoheadrightarrow \mcF_i^{[\N_i-1]} \twoheadrightarrow \cdots \twoheadrightarrow \mcF_i^{[1]} \twoheadrightarrow \mcF_i^{[0]} = 0
\end{align}
for some  $\N_i \in \mbZ_{>0}$.
Parabolic bundles and their moduli spaces  were introduced by V. B. Mehta and C. S. Seshadri (cf. ~\cite{MeSe}), and later investigated by many mathematicians  from various points of view, including  conformal field theory (cf., e.g.,  ~\cite{BMW}, ~\cite{Kum}).

Let us now consider  vector bundles in characteristic $p> 0$.
Note that some remarkable phenomena occurring  in characteristic $p$  are  derived from   properties of the Frobenius morphisms, which are endomorphisms given by sending every  regular function  to its  $p$-th power.
For example,
a classical result by Cartier (cf. ~\cite[Theorem 5.1]{Kal})  asserts that the pull-back along the relative Frobenius morphism $F_{X/k} : X \rightarrow X^{(1)}$ from $X$ to  its  Frobenius twist $X^{(1)}$ over $k$ yields an equivalence of categories
\begin{align} \label{EQ599}
\left(\begin{matrix} \text{the category of } \\ \text{vector bundles on $X^{(1)}$} 
\end{matrix} \right)
\xrightarrow{\sim} 
\left(\begin{matrix} \text{the category of} \\ \text{$p$-flat bundles on $X$} 
\end{matrix} \right),
\end{align}
where a {\it $p$-flat bundle} means a vector bundle equipped with (flat) connection with vanishing $p$-curvature (cf. ~\cite[\S\,5]{Kal} for the definition of $p$-curvature).
 This equivalence can be 
  generalized  to 
 ``{\it higher levels}"
  by using sheaves  of differential operators of finite level, introduced by P. Berthelot (cf. ~\cite{PBer1}, ~\cite[Corollary 3.2.4]{LeQu}).
Moreover, 
we know some extensions  of such  equivalences, known as the Ogus-Vologodsky correspondence,  by which 
nilpotent Higgs bundles corresponds to  
flat bundles
with nilpotent $p$-curvature (cf. ~\cite{GLQ}, ~\cite{Ohk}, ~\cite{OgVo}, ~\cite{Sch}).

In the present paper, we address the issue of
extending various results and arguments obtained for the non-parabolic case, including Cartier's theorem, by appropriately defining ``{\it Frobenius pull-back}" of parabolic bundles on pointed curves.
For positive integers $\NN$ and  $\N$,  denote by $\Xi_{\N, \NN}^{\leq}$ (cf. \eqref{EQ401})
the set of $\N$-tuples of integers $(a^{[1]}, \cdots, a^{[\N]})$ with $0 \leq a^{[1]} \leq  \cdots \leq  a^{[\N]} < p^\NN$. 
Let $(\N_1, \cdots, \N_r)$ be an $r$-tuple of positive integers 
and  $\vec{a} := (a_1, \cdots, a_r)$ an $r$-tuple such that
$a_i$ (for each $i=1, \cdots, r$) denotes an element $(a_i^{[1]}, \cdots, a_i^{[\N_i]})$ of $\Xi_{\N_i, \NN}^{\leq}$.
Also, we set $\vec{a} := (a_1/p^\NN, \cdots, a_r/p^\NN)$, where
$a_i/p^\NN := (a_i^{[1]}/p^\NN, \cdots, a_i^{[\N_i]}/p^\NN)$.
For a parabolic bundle of the form $\msE := (\mcE, \vec{\mfe}, \vec{a}/p^\NN)$ on the $\NN$-th Frobenius twist 
$\msX^{(\NN)} := (X^{(\NN)}, \{ \sigma_i^{(\NN)}\}_{i=1}^r)$   of $\msX$,
 the pull-back of $\msE$ along  the $\NN$-th relative Frobenius morphism $F_{X/k}^{(\NN)} : X \rightarrow X^{(\NN)}$ yields 
  a $p^\NN$-flat bundle (i.e., a parabolic and level-$\NN$ enrichment of  flat bundle, in the sense of Definition \ref{Def66}, (i))
\begin{align}
\msE^F_\flat := (\mcE^F, \nabla^F, \vec{\mfe}^{\,F}, \vec{a})
\end{align}
on $\msX$ (cf. \eqref{EQ167}); this will be called the {\it parabolic Frobenius pull-back} of $\msE$ by $F_{X/k}^{(\NN)}$ (cf. Definition \ref{Def113}).
(The pull-backs of parabolic bundles by a ramified   covering in characteristic zero are discussed  in 
~\cite{BKP}.
Our study can be regarded as  its characteristic-$p$ analogue for the case of Frobenius morphisms.)
Then, we prove the following assertion generalizing Cartier's theorem.

\begin{intthm}[cf. Theorem \ref{Prop11}]
 \label{ThC}
Denote by  $\mcB un_{\msX^{(\NN)}, \vec{a}/p^\NN}$ (resp., $\mcD^{(\NN -1)}\text{-}\mcB un_{\msX, \vec{a}}^{\psi = 0}$)
the category of parabolic bundles on $\msX^{(\NN)}$ of parabolic weights $\vec{a}/p^\NN$ (resp., parabolic $p^\NN$-flat bundles on $\msX$ of parabolic weights $\vec{a}$).
Then,
the assignment $\msE \mapsto \msE^F_\flat$
 gives an equivalence of categories
\begin{align} \label{EQ11171}
\mcB un_{\msX^{(\NN)}, \vec{a}/p^\NN} \xrightarrow{\sim}
\mcD^{(\NN -1)}\text{-}\mcB un_{\msX, \vec{a}}^{\psi = 0}.
\end{align}
\end{intthm}

The second part of the present paper concerns  the stability of parabolic bundles.
For a parabolic bundle $\msF$ as in \eqref{EQ646},
its parabolic  slope  is defined as the value
\begin{align}
\mr{par}\text{-}\mu(\msF) :=
\frac{1}{\mr{rk}(\mcF)} \cdot \left(\mr{deg}(\mcF) + \sum_{i=1}^r \sum_{j} a_i^{[j]} \LL_i^{[j]} \right),
\end{align}
where $a_i^{[j]}$'s and $\LL_i^{[j]}$'s are the parabolic weights and types, respectively,  of the parabolic structure $(\vec{\mff}, \vec{\alpha})$ at $\sigma_i$ 
(cf. \eqref{EQ299}).
Then, we say that $\msF$ is {\it semistable} (resp., {\it stable}) if the inequality $\mr{par}\text{-}\mu (\msE) \leq \mr{par}\text{-}\mu (\msF)$ (resp., $\mr{par}\text{-}\mu (\msE) < \mr{par}\text{-}\mu (\msF)$) holds for any nontrivial  proper parabolic subbundle $\msE$ of $\msF$ (cf. Definition \ref{Def44}).

We are primarily interested in stable parabolic bundles whose parabolic Frobenius pull-back by $F_{X/k}^{(\NN)}$ are maximally unstable with respect to the Harder-Narasimhan polygon; these bundles are called {\it maximally $F^{(\NN)}$-destabilized (parabolic) bundles} (cf. Definition \ref{Def}).
For  the previous study of the non-parabolic case,
we refer the reader to  ~\cite{JoPa}, ~\cite{JRXY}, 
~\cite{Li}, and ~\cite{Zha}.
By a result of K. Joshi and C. Pauly (cf. ~\cite[Theorem 5.3.1]{JoPa}), 
 the equivalence of categories \eqref{EQ599} restricts to a bijective correspondence between maximally Frobenius-destabilized
   bundles and  certain $p$-flat bundles called dormant opers (cf. ~\cite{JRXY}, ~\cite{JoPa}, ~\cite{Jos}, ~\cite{Mzk2}, ~\cite{Wak1}, ~\cite{Wak5}, and ~\cite{Wak2} for the previous study of dormant opers).

To state our result concerning its parabolic (and higher-level) generalization, we prepare some notations.
Let $n$ be a positive integer, $\mcL$  a line bundle on $X^{(\NN)}$,   and $\vec{a}$ an element  of $(\Xi_{n, \NN}^{<})^{\times r}$ as in \S\,\ref{SS311} such that $s_1^1 (\vec{a}) \in (\Xi_{n, 1}^{\leq})^{\times r}$
 and $s_2^1 (\vec{a}) \in (\Xi_{n, \NN -1}^{\leq})^{\times r}$ (cf. \eqref{EQ503} for the definitions of $s_1^{1}(-)$ and $s_2^1 (-)$).
 We set $\vec{1}_{\times n} := (1_{\times n}, \cdots, 1_{\times n})$ (= the $r$-tuple of copies of $1_{\times n} := (1, \cdots, 1) \in \Xi_{n, \NN}^{\leq}$), $\omega := (n, \vec{1}_{\times n}, \vec{a}/p^\NN)$, and 
$\mcL^F_\flat := ((F^{(\NN)*}_{X/k}\mcL)(D_\star), \nabla^\mr{can}_{\mcL, D_\star})$ (cf. \eqref{Eq214}),
where $D_\star$ denotes the divisor $\sum\limits_{i=1}^r ( \sum\limits_{j=1}^n a_i^{[j]}) \sigma_i$.
 Then,  we obtain  the moduli  category
\begin{align} \label{EQ178}
\mcU^F_{\msX^{(\NN)}, \omega, \mcL} \ \left(\text{resp.,} \ \mcO p^{^\mr{Zzz...}}_{\msX, n, \vec{a}, \mcL_\flat^F} \right)
\end{align}
(cf. \eqref{Eq405}, \eqref{EQ178})  classifying
 maximally $F^{(\NN)}$-destablized stable  parabolic bundles on $\msX^{(\NN)}$ of rank $n$, type $\vec{1}_{\times n}$, and parabolic weights $\vec{a}/p^\NN$ (resp., dormant $\mr{GL}_n^{(\NN)}$-opers on $\msX$ of exponent $\vec{a}$ with determinant $\mcL_\flat^F$).

\begin{intthm}[cf. Theorem \ref{Prop3}, (i) and (ii)] \label{ThA}
Suppose that the following  three conditions are fulfilled:
\begin{itemize}
\item
 $n \mid \mr{deg}(\mcL)$;
 \item
 $\sum\limits_{j=1}^n a_i^{[j]} < p^\NN$ for every $i=1, \cdots, r$;
 \item
 $\sum\limits_{i=1}^r (s_1^1(a_i^{[n]})-s_1^1(a_i^{[1]})) < \frac{n(2g-2+r)}{2} < \frac{p}{n}$.
 \end{itemize}
Then,
there is a natural  
isomorphism of fibered categories 
\begin{align}
\mcU^F_{\msX^{(\NN)}, \omega, \mcL} \xrightarrow{\sim} \mcO p_{\msX, n, \vec{a}, \mcL^F_\flat}^{^\mr{Zzz...}}. 
\end{align}
 In particular, there are only finitely many isomorphism classes of maximally $F^{(\NN)}$-destabilized stable  parabolic bundles on $\msX^{(\NN)}$ 
 of rank $n$, type $\vec{1}_{\times n}$, and parabolic weights $\vec{a}/p^\NN$.
\end{intthm}

As an application of this result and the study on dormant opers discussed in ~\cite{Wak2}, 
we can give an explicit formula for computing the number of maximally $F^{(1)}$-destabilized parabolic bundles of rank $2$.
The assertion is described as follows.

\begin{intthm}[cf. Theorem \ref{Th22} for the full statement] \label{ThB}
Let us take a line bundle $\mcL$ on $X^{(1)}$ of even degree and   an element  $\vec{a} := (a_1, \cdots, a_r) \in (\Xi_{2, 1}^{<})^{\times r}$, where $a_i := (a_i^{[1]}, a_i^{[2]})$.
We write $\omega := (2, \vec{1}_{\times 2}, \vec{a}/p)$.
Suppose further that  the following  two conditions are fulfilled:
\begin{itemize}
\item
$r + \sum\limits_{i=1}^r (a_i^{[1]}+ a_i^{[2]})$ is even and $\sum\limits_{i=1}^r (a_i^{[2]}-a_i^{[1]}) <2g-2+r <  \frac{p}{2}$;
\item
$\msX$ is sufficiently general in the moduli stack $\mcM_{g, r}$ of $r$-pointed smooth proper curves of genus $g$. 
\end{itemize}
Then, the  cardinality $\sharp (\mcU^F_{\msX^{(1)}, \omega, \mcL} (k))$ of the set of isomorphism classes of parabolic bundles classified by   $\mcU^F_{\msX^{(1)}, \omega, \mcL} (k)$ satisfies the following equality:
 \begin{align}
  \sharp (\mcU^F_{\msX^{(1)}, \omega, \mcL} (k)) 
  =   2 \cdot p^{g-1} \cdot  \sum_{j =1}^{p-1} \frac{\prod\limits_{i=1}^r (-1)^{(j+1)(a_i^{[2]}-a_i^{[1]}+1)} \sin \left( \frac{(a_i^{[2]}-a_i^{[1]}) j \pi}{p}\right)}{\sin^{2g-2+r} \left(\frac{j \pi}{p} \right)}.
  \end{align}
\end{intthm}

Finally, we
remark that  parabolic bundles in positive characteristic were  previously  investigated  in 
~\cite{KuMa} and ~\cite{KuPa} (resp., ~\cite{LiSu} and ~\cite{She}) from a viewpoint of Galois coverings (resp., non-Abelian Hodge correspondence).
It should be a future task to understand how can we relate  that work to our discussion concerning parabolic Frobenius pull-back.

\subsection*{Notation and Conventions} 
Throughout the present paper, we fix a prime number $p$, 
 an algebraically closed field $k$ of characteristic $p$,  and a pair of nonnegative integers $(g, r)$.
Also, we fix a $k$-scheme  $S$ and    an $r$-pointed  smooth  curve
\begin{align}
\msX := (f : X \rightarrow S, \{ \sigma_i :S \rightarrow X\}_{i=1}^r)
\end{align}
   of genus $g$ over $S$, i.e.,   
$X$ is   a  smooth proper curve  over $S$ of genus $g$ and $\{ \sigma_i \}_{i=1}^r$ denotes  a collection of mutually disjoint $r$ marked points.
If $S'$ is an $S$-scheme, then we denote by $\msX_{S'}$ the base-change of $\msX$ over $S'$.

Denote by $D$ the effective relative divisor on $X$ (relative to $S$) defined as the union of the marked points $\{ \sigma_i \}_{i=1}^r$;
it defines a log structure on $X$, and the resulting  log scheme will be denoted by $X^\mr{log}$.

Since $X^\mr{log}/S$ is log smooth, the sheaf of logarithmic $1$-forms $\Omega_{X^\mr{log}/S}$ of $X^\mr{log}/S$, as well as its dual $\mcT_{X^\mr{log}/S} := \Omega_{X^\mr{log}/S}^\vee$, is a line bundle.
In fact, we have natural isomorphisms $\Omega_{X^\mr{log}/S} \cong \Omega_{X/S} (D)$, $\mcT_{X^\mr{log}/S} \cong \mcT_{X/S}(-D)$.
When there is no fear of confusion, we write $\Omega$ and $\mcT$ instead of $\Omega_{X^\mr{log}/S}$ and $\mcT_{X^\mr{log}/S}$, respectively. 
Also,  for  a positive integer $\NN$, we  denote by $\mcD^{(\NN -1)}_{X^\mr{log}/S}$ (cf. ~\cite[Definition 2.3.1]{Mon}) the sheaf of  logarithmic differential operators on $X^\mr{log}/S$ 
 of level $\NN -1$.

Next, write $F_S$ (resp., $F_X$) for the absolute Frobenius endomorphism of $S$ (resp., $X$).
For each positive integer $\NN$, we shall denote by $X^{(\NN)}$ the base-change  $S \times_{F_S^\NN, S} X$ of $X$ along  the $\NN$-th iterate $F_S^\NN$ of $F_S$; we will refer to it as the {\bf $\NN$-th Frobenius twist} of $X$ over $S$.
Also, the morphism $F_{X/S}^{(\NN)}  \left(:= (f, F_X^\NN) \right) : X \rightarrow X^{(\NN)}$ is called the {\bf $\NN$-th relative Frobenius morphism} of $X$ over $S$.

\section{Parabolic vector bundles} \label{S1}

In this section, we recall parabolic bundles on an algebraic curve and the stability for such objects.
There are many references discussing parabolic bundles; for example, we refer the reader to ~\cite{MaYo}, ~\cite{MeSe}. 

\subsection{Definition of a parabolic bundle} \label{SS1}

In Definitions \ref{Def20} and \ref{Def19} below,
we shall fix a vector bundle  $\mcF$ on $X$ of constant rank.

\bde\label{Def20}
\begin{itemize}
\item[(i)]
Let  $i$ be  an element of 
$\{1, \cdots, r \}$, and 
we shall write $\mcF_i := \sigma_{i}^*\mcF$, which forms a vector bundle on $S$.
By a {\bf quasi-parabolic structure} of $\mcF$ at $\sigma_{i}$, we mean a sequence of $\mcO_S$-linear surjections
\begin{align} \label{Eq113}
\mff_{i} : \mcF_i  =   \mcF_{i}^{[\N_i]} 
\xrightarrow{\mff_i^{[\N_i]}}
\mcF_{i}^{[\N_i-1]} 
\xrightarrow{\mff_i^{[\N_i-1]}}
 \cdots 
 \xrightarrow{\mff_i^{[2]}}
 \mcF_{i}^{[1]} 
 \xrightarrow{\mff_i^{[1]}}
  \mcF_{i}^{[0]}=0
\end{align} 
(where $\N_i \in \mbZ_{>0}$)
between vector bundles $\mcF_i^{[j]}$ on $S$.
A {\bf parabolic structure} of $\mcF$ at $\sigma_i$   is a pair 
\begin{align} \label{Eq110}
(\mff_i, \alpha_i)
\end{align}
  consisting of a quasi-parabolic structure $\mff_i$ as in \eqref{Eq113}  and an $\N_i$-tuple of real numbers $\alpha_i :=  (\alpha^{[1]}_i, \cdots, \alpha^{[\N_i]}_i)$ with $0 \leq \alpha^{[1]}_i \leq   \cdots \leq  \alpha^{[\N_i]}_i$.
   \item[(ii)]
   A {\bf quasi-parabolic} (resp., {\bf  parabolic}) {\bf structure} of $\mcF$ is a collection of data
  \begin{align} \label{Eq109}
\vec{\mff} \ \left(\text{resp.,} \  (\vec{\mff}, \vec{\alpha})\right)
  \end{align}
  consisting of 
  an $r$-tuple $\vec{\mff} :=(\mff_1, \cdots, \mff_r)$ (resp., 
  $r$-tuples  $\vec{\mff} :=(\mff_1, \cdots, \mff_r)$ and $\vec{\alpha} := (\alpha_1, \cdots, \alpha_{r})$)
  such that, for each $i=1, \cdots, r$,
  $\mff_i$ (resp., $(\mff_i, \alpha_i)$) forms a quasi-parabolic (resp., parabolic) structure  of $\mcF$ at $\sigma_i$.
\end{itemize}
\ede

\bde \label{Def19}
\begin{itemize}
\item[(i)]
Let $i$ be an element of $\{1, \cdots, r \}$ and  $\mff_i$ (resp., $(\mff_i, \alpha_i)$)  a quasi-parabolic  (resp., parabolic)  structure of $\mcF$ at $\sigma_{i}$ as in \eqref{Eq110}.
Then, 
the integer $\N_i$ is called the {\bf length} of  $\mff_i$ (resp., $(\mff_i, \alpha_i)$).
Also, 
if $S$ is connected, 
then we obtain 
the $\N_i$-tuple of nonnegative integers 
\begin{align} \label{Eq210}
\LL_i := (\LL_i^{[1]}, \cdots, \LL_i^{[\N_i]}),
\end{align}
where $\LL_i^{[j]}$ denotes the rank of the vector bundle
$\mr{Ker}(\mff_i^{[j]}
)$, and it 
is called the {\bf type} of $\mff_i$ (resp., $(\mff_i, \alpha_i)$).
\item[(ii)]
Let $\vec{\mff}$ (resp., $(\vec{\mff}, \vec{\alpha})$), where $\vec{\mff} := (\mff_1, \cdots, \mff_r)$ (resp., $\vec{\mff} := (\mff_1, \cdots, \mff_r)$ and $\vec{\alpha} := (\alpha_1, \cdots, \alpha_r)$),  be a quasi-parabolic (resp., parabolic) structure of $\mcF$.
Under the assumption that $S$ is connected, we
denote by $\LL_i$ the type of $\mff_i$ (resp., $(\mff_i, \alpha_i)$).
Then, the collections
\begin{align} \label{Eq245}
\vec{\N} := (\N_1, \cdots, \N_r)  \ \text{and} \  \vec{\LL} := (\LL_1, \cdots, \LL_r) 
\end{align}
are called  the {\bf length} and  {\bf type} of $\vec{\mff}$ (resp., $(\vec{\mff}, \vec{\alpha})$), respectively.
\end{itemize}
\ede

\bde \label{Def4453} 
\begin{itemize}
\item[(i)]
A {\bf quasi-parabolic} (resp.,  {\bf parabolic}) {\bf bundle} 
 on $\msX$ is a collection of data
\begin{align} \label{Eq201}
\msF := (\mcF, \vec{\mff}) \ \left(\text{resp.,} \ 
\msF := (\mcF,  \vec{\mff},  \vec{\alpha})\right),
\end{align}
consisting of 
 a vector bundle  $\mcF$  on $X$ of constant rank  and
a quasi-parabolic (resp., parabolic) structure $\vec{\mff}$ (resp., $(\vec{\mff}, \vec{\alpha})$)
 of $\mcF$.
If  $\mff_i := \{ \mff_i^{[j']}\}_{j'=1}^{\N_i}$ denotes 
the $i$-th factor of $\vec{\mff}$ (where  $i \in \{1, \cdots, r \}$),
then,
for each $j =1, \cdots, \N_i$,
 we 
obtain a composite surjection
\begin{align} \label{Eq202}
\varpi_{\msF, i}^{[j]} :  \mcF \twoheadrightarrow  \sigma_{i*}\mcF_{i}^{[\N_i]} 
\xrightarrow{\sigma_{i*}\mff_i^{[\N_i]}} \cdots  \xrightarrow{\sigma_{i*}\mff_i^{[j+1]}} 
\sigma_{i*}\mcF_{i}^{[j]}.
\end{align}
\item[(ii)]
Let 
$\msF$
 be a (quasi-)parabolic bundle on $\msX$ as in \eqref{Eq201}.
We shall say that $\msF$ is  {\bf of  rank $n (> 0)$}
if the vector bundle $\mcF$ has constant rank $n$.
 Also, in the resp'd portion of (i), the numbers in $\vec{\alpha}$ are  called the {\bf  parabolic weights} of $\msF$.
\end{itemize}
\ede

\begin{rem} \label{Rem100}
For convenience of our discussion (cf. Definition \ref{Def66}, (i)),
we do not impose (unlike the usual definition) the condition that the parabolic weights of a parabolic bundle  is less than $1$.
\end{rem}

Next, let  $\msF := (\mcF,  \vec{\mff}, \vec{\alpha})$ and $\msE := (\mcE,  \vec{\mfe}, \vec{\alpha})$ be two parabolic  bundles on $\msX$ with parabolic structures  having   common  length and parabolic weights 

\bde\label{Def31}
A {\bf morphism of parabolic bundles} from $\msE$ to $\msF$ is an $\mcO_{X}$-linear morphism $h : \mcE \rightarrow \mcF$ that restricts to   a morphism   $\mr{Ker} (\varpi_{\msE, i}^{[j]}) \rightarrow \mr{Ker}(\varpi_{\msF, i}^{[j]})$ for every $i$ and $j$.
\ede

Let us choose 
an $r$-tuple    $\vec{\alpha} := (\alpha_1, \cdots, \alpha_r)$
such that $\alpha_i$ (for each $i=1, \cdots, r$) is a collection  of real numbers $(\alpha_i^{[1]}, \cdots, \alpha_i^{[\N_i]})$ ($\N_i \in \mbZ_{> 0}$) with $0 \leq \alpha_i^{[1]} \leq  \cdots \leq  \alpha_i^{[\N_i]}$.
This $r$-tuple  determines  the category
\begin{align} \label{EQ135} 
\mcB un_{\msX, \vec{\alpha}}
\end{align}
 consisting of 
 parabolic bundles on $\msX$  
 having (length $(\N_1, \cdots, \N_r)$ and)  parabolic  weights  $\vec{\alpha}$.

\subsection{Stability for parabolic   bundles} \label{SS51}

Let $\msF := (\mcF,  \vec{\mff},  \vec{\alpha})$ be a parabolic  bundle
 on $\msX$ such that $(\vec{\mff},  \vec{\alpha})$ is a collection as in the resp'd portion of \eqref{Eq109}. 
Given a subbundle $\mcE$ of $\mcF$ of constant rank,
we set  $\mcF^{[j]}_{i, \mcE}$ (for each $i =1, \cdots, r$ and $j = 1, \cdots, \N_i$) to  be the image of $\mcE$ via $\varpi_{\msF, i}^{[j]}$.
Assume that   $\mcF^{[j]}_{i, \mcE}$ are all vector bundles on $S$ when restricted via $\sigma_i$.
(Of course, this assumption is automatically satisfied if $S = \mr{Spec}(k)$.)
Since
 $\mff_i^{[j]}$ restricts to a surjection
 $\mff_{i, \mcE}^{[j]} : \mcF^{[j]}_{i, \mcE} \twoheadrightarrow \mcF^{[j-1]}_{i, \mcE}$,  the set of morphisms $\mff_{i, \mcE} := \{ \mff_{i, \mcE}^{[j]} \}_{j=1}^{\N_i}$ specifies, via $\sigma_i^*(-)$, a quasi-parabolic structure on $\mcE$ at $\sigma_i$.
The  resulting collection  $\msE := (\mcE, \vec{\mff}_\mcE, \vec{\alpha})$, where
$\vec{\mff}_\mcE := (\mff_{1, \mcE}, \cdots, \mff_{r, \mcE})$,
forms a parabolic bundle on $\msX$;
such a parabolic bundle 
obtained  in this way is called {\bf parabolic subbundle} of $\msF$.

Suppose that $S$ is connected.
Each parabolic subbundle $\msE$ as above determines 
the following two real numbers:
\begin{align} \label{EQ299}
\mr{par}\text{-}
\mr{deg} (\msE) = \mr{deg} (\mcE) +
\sum_{i=1}^r\sum_{j=1}^{\N_i} \alpha_i^{[j]} \LL_{i, \mcE}^{[j]},
\hspace{8mm}
\mr{par}\text{-}
\mu (\msE) := \frac{\mr{par}\text{-}\mr{deg}(\msE)}{\mr{rk}(\mcE)},
\end{align}
where $\mr{deg}(\mcE)$ denotes the relative degree of $\mcE$ and 
$\LL^{[j]}_{i, \mcE}$ (for each $i$ and $j$) denotes
the rank of  $\mr{Ker}(\mff^{[j]}_{i, \mcE})$ (viewed as a vector bundle on $S$).
We refer to  $\mr{par}\text{-}\mr{deg} (\msE)$ and $\mr{par}\text{-}\mu (\msE)$ as, respectively,  the {\bf parabolic degree} and the 
 {\bf parabolic slope} of $\msE$.
 The formations of  
 $\mr{par}\text{-}\mr{deg} (-)$ and $\mr{par}\text{-}\mu (-)$ commute  with the base-change  over any $S$-scheme.

\bde \label{Def44}
We shall  say  that $\msF$ is {\bf semistable} (resp., {\bf stable}) 
 if for any geometric point $t$ of $S$ and any nontrivial proper parabolic  subbundle $\msE_t$ of the fiber $\msF_t$ of $\msF$ over $t$,
 one has
  \begin{align}
  \mr{par}\text{-}
  \mu (\msE_t) \leq 
  \mr{par}\text{-}
  \mu (\msF_t) \  \left(\text{resp.,}  \ 
  \mr{par}\text{-}
  \mu (\msE_t) < 
  \mr{par}\text{-}
  \mu (\msF_t)\right).
  \end{align}
\ede

Let $\omega := (n, \vec{\LL}, \vec{\alpha})$ be a triple  consisting of a positive integer $n$, 
an $r$-tuple    $\vec{\LL} := (\LL_1, \cdots, \LL_r)$ such that $\LL_i$ (for each $i=1, \cdots, r$)
is a sequence of nonnegative integers $(\LL_i^{[1]}, \cdots, \LL_i^{[\N_i]})$ (where $\M_i \in \mbZ_{> 0}$),
and 
a collection of parabolic weights $\vec{\alpha} := (\alpha_1, \cdots, \alpha_r)$ as before.
Then, for  a line bundle $\mcL$ on $X$,
we obtain  the category  fibered in groupoids 
\begin{align} \label{EQ322}
\overline{\mcU}_{\msX, \omega, \mcL} \ \left(\text{resp.,} \ \mcU_{\msX, \omega, \mcL} \right)
\end{align}
over the category of $S$-schemes  $\mcS ch_{/S}$  defined as follows:
the fiber  $\overline{\mcU}_{\msX, \omega, \mcL}(S')$ (resp., $\mcU_{\msX, \omega, \mcL}(S')$) over each $S$-scheme $s : S' \rightarrow S$ is 
 the groupoid of  pairs $(\msF, \eta)$,
 where
 \begin{itemize}
 \item
 $\msF$ denotes  a rank-$n$ semistable (resp., stable) parabolic bundle on $\msX_{S'}$   whose parabolic structure 
 has type $\vec{\LL}$ and
  parabolic weights $\vec{\alpha}$;
 \item
 $\eta$ denotes an isomorphism $\mr{det}(\mcF) \xrightarrow{\sim} \mcL$, where $\mcF$ is the underlying vector bundle of $\msF$.
 \end{itemize}

\section{Parabolic flat bundles of higher level} \label{S51}

This section deals with a higher-level generalization of $\mcD$-modules in positive characteristic.
This concept was originally studied in ~\cite{PBer1} and later generalized in ~\cite{Mon} to the logarithmic situation.
We here recall monodromy operators  (in other words, residues) and exponents of higher level, in the manner of  
~\cite{Wak4}; the exponents of a higher-level   $\mcD$-module  describe  its local structures around  the marked points.

\subsection{Log flat bundles and higher-level $\mcD$-modules} \label{SS45}

Let us fix a positive integer $\NN$.
For simplicity, we use the notation
 ``$\mcD^{(\NN -1)}$"  to denote the sheaf 
 $\mcD^{(\NN -1)}_{X^\mr{log}/S}$.
For each integer $j$, we shall write $\mcD^{(\NN -1)}_{< j}$ for the subsheaf  of $\mcD^{(\NN -1)}$ consisting of logarithmic differential operators of  order $< j$.

Note that $\mcD^{(\NN -1)}$ is endowed with a structure of $\mcO_X$-module arising from left (resp., right) multiplication by sections of $\mcD_{<1}^{(\NN -1)} \left(= \mcO_X \right)$; we denote the resulting $\mcO_X$-module by ${^L}\mcD^{(\NN -1)}$ (resp., ${^R}\mcD^{(\NN -1)}$).
A {\bf (left) $\mcD^{(\NN -1)}$-module structure} on an $\mcO_X$-module $\mcF$ is a left $\mcD^{(\NN -1)}$-action $\nabla : {^L}\mcD^{(\NN -1)} \rightarrow \mcE nd_{f^{-1}(\mcO_S)} (\mcF)$ on $\mcF$ extending its $\mcO_X$-module structure.
An $\mcO_X$-module equipped with a $\mcD^{(\NN -1)}$-module structure is called a {\bf (left) $\mcD^{(\NN -1)}$-module}.

 Given a $\mcD^{(\NN -1)}$-module $(\mcF, \nabla)$, we shall write 
 $\mcS ol (\nabla)$
   for the subsheaf of $\mcF$ on which $\mcD_{+}^{(\NN -1)}$ acts as zero, where $\mcD_+^{(\NN -1)}$ denotes the kernel of the canonical projection $\mcD^{(\NN -1)} \twoheadrightarrow \mcO_X$.
 The sheaf $\mcS ol (\nabla)$ may be regarded as an $\mcO_{X^{(\NN)}}$-module via the underlying homeomorphism of $F_{X/S}^{(\NN)}$.

 Denote by $\psi$ the $p^\NN$-curvature map $\mcT^{\otimes p^\NN} \rightarrow \mcD^{(\NN -1)}$ defined in ~\cite[\S\,3.2, Definition 3.10]{Ohk}.
 Given a $\mcD^{(\NN -1)}$-module $(\mcF, \nabla)$, we shall set
 \begin{align}
 \psi_{(\mcF, \nabla)}  := \nabla \circ \psi : \mcT^{\otimes p^\NN} \rightarrow \mcE nd_{f^{-1}(\mcO_S)} (\mcF),
 \end{align}
 which will be  called the {\bf $p^\NN$-curvature} of $(\mcF, \nabla)$.
 It is well-known that the image of  $\psi_{(\mcF, \nabla)}$  lies in $\mcE nd_{\mcO_X} (\mcF)$.

\bde \label{Def4981}
 By a {\bf $p^\NN$-flat bundle} on $\msX$, we shall mean a pair $(\mcF, \nabla)$ consisting of a vector bundle $\mcF$ on $X$ and a $\mcD^{(\NN -1)}$-module structure $\nabla$ on $\mcF$ with vanishing $p^\NN$-curvature.
\ede

Let 
 $\mcE$ be an $\mcO_{X^{(\NN)}}$-module.
 According to  ~\cite[Corollaire 3.3.1]{Mon}, 
 there exists a  canonical $\mcD^{(\NN -1)}$-module structure
 \begin{align} \label{Eq204}
 \nabla_{\mcE}^\mr{can} : {^L}\mcD^{(\NN -1)} \rightarrow \mcE nd_{f^{-1}(\mcO_S)} ( F^{(\NN)*}_{X/S} \mcE)
 \end{align}
on the $\NN$-th Frobenius pull-back $F^{(\NN)*}_{X/S}\mcE$ uniquely determined by the condition that
every local section $(F^{(\NN)}_{X/S})^{-1}(v)$ (for $v \in \mcE$) is horizontal.
Moreover, for a relative divisor 
$D_0$ on $X/S$ with $\mr{Supp} (D_0) \subseteq \mr{Supp} (D) \left(= \bigcup_{i=1}^r \mr{Im}(\sigma_i) \right)$,
one can obtain  a $\mcD^{(\NN -1)}$-module structure
\begin{align} \label{Eq214}
 \nabla_{\mcE, D_0}^\mr{can} : {^L}\mcD^{(\NN -1)} \rightarrow \mcE nd_{f^{-1}(\mcO_S)} ( (F^{(\NN)*}_{X/S} \mcE) (D_0))
 \end{align}
 on $(F^{(\NN)*}_{X/S} \mcE) (D_0) \left(=  F^{(\NN)*}_{X/S} \mcE \otimes \mcO_X (D_0)\right)$ that is consistent with   $\nabla_{\mcE}^\mr{can}$ outside $\mr{Supp}(D_0)$.
It is immediately verified that $\nabla^\mr{can}_\mcE$ and $\nabla_{\mcE, D_0}^\mr{can}$ have  vanishing $p^\NN$-curvature.

\subsection{Monodromy operators} \label{SS69}

We shall set
\begin{align} \label{YY11}
\mcB_S := \bigoplus_{j \in \mbZ_{\geq 0}} \mcO_S \cdot \partial_\mcB^{\langle j \rangle},
\end{align}
where $\partial_\mcB^{\langle i \rangle}$'s are   abstract symbols (cf. ~\cite[\S\,4.2]{Wak4}).
  Here,  for  each  $j \in \mbZ_{\geq 0}$, let
   $q_j$ denote a unique nonnegative integer satisfying 
      $j = p^{\NN -1} \cdot q_j + r_j$ and $0 \leq r_j < p^{\NN -1}$.
Then, we equip $\mcB_S$ with a structure of $\mcO_S$-algebra given by 
\begin{align} \label{YY12}
\partial_\mcB^{\langle j' \rangle} \cdot \partial_\mcB^{\langle j''\rangle} := \sum_{j = \mr{max}\{j', j'' \}}^{j' + j''} \frac{j!}{(j' + j''-j)! \cdot (j-j')! \cdot (j-j'')!}\cdot \frac{q_{j'}! \cdot q_{j''}!}{q_j !} \cdot \partial_\mcB^{\langle j\rangle}.
\end{align}
The  $\mcO_S$-algebra $\mcB_S$ is commutative and  generated by the  global sections $\partial_\mcB^{\langle 1 \rangle}, \partial_\mcB^{\langle p \rangle}, \cdots, \partial_\mcB^{\langle p^{\NN -1}\rangle}$.

Let $\mcG$ be an $\mcO_S$-module and $\mu$ a morphism of $\mcO_S$-algebras 
$\mcB_S \rightarrow \mcE nd_{\mcO_S} (\mcG)$.
For each $a \in \mbZ_{\geq 0}$,
we set 
$\mu^{\langle p^a \rangle} := \mu  (\partial_\mcB^{\langle p^a \rangle})$, which is an element of $\mr{End}_{\mcO_S}(\mcG)$.
The morphism $\mu$ is uniquely determined by  
the $\NN$-tuple
\begin{align} \label{e187}
\mu^{\langle \bullet \rangle} := (\mu^{\langle 1 \rangle}, \mu^{\langle p \rangle} \cdots,  \mu^{\langle p^{\NN -1} \rangle}) \in \mr{End}_{\mcO_S}(\mcG)^{\oplus \NN}.
\end{align} 

Next, let $(\mcF, \nabla)$ be a $\mcD^{(\NN -1)}$-module, and choose  $i \in \{ 1, \cdots, r \}$.
Note that there exists a canonical isomorphism of $\mcO_S$-algebras
$\sigma_i^* ({^L}\mcD^{(\NN -1)}) \xrightarrow{\sim} \mcB_S$ such that
the local basis 
$\{ \partial^{\langle j \rangle}\}_{j \in \mbZ_{\geq 0}}$ of ${^L}\mcD^{(\NN -1)}$ associated to 
 any logarithmic coordinate around  a point  in $\mr{Im}(\sigma_i)$
 is mapped to $\{ \partial_\mcB^{\langle j \rangle} \}_{j \in \mbZ_{\geq 0}}$ (cf. ~\cite[Lemma 2.2.2 and \S\,6.1]{Wak4}).
The  $\mcD^{(\NN -1)}$-module structure  $\DMO$ induces  an $\sigma_i^*({^L}\mcD^{(\NN -1)})$-action $\sigma_i^*\nabla$  on $\sigma_i^*\mcF$; it gives 
the composite
\begin{align} \label{EQ400}
\mu_i (\nabla) : \mcB_S \xrightarrow{\sim} \sigma_i^*({^L}\mcD^{(\NN -1)}) \xrightarrow{\sigma_i^* \nabla} \mcE nd_{\mcO_S} (\sigma_i^* \mcF),
\end{align}
and  the pair $(\sigma_i^*\mcF, \mu_i (\nabla))$ specifies a $\mcB_S$-module. 
In particular, we obtain 
\begin{align} \label{e18d7}
\mu_i (\DMO)^{\langle \bullet \rangle} := (\mu_i (\DMO)^{\langle 1 \rangle}, \mu_i (\DMO)^{\langle p \rangle} \cdots,  \mu_i (\DMO)^{\langle p^{\NN -1} \rangle}) \in \mr{End}_{\mcO_S}(\sigma_i^*\mcF)^{\oplus \NN}.
\end{align} 
Since $\mcB_S$ is commutative,
the elements $\mu_i (\DMO)^{\langle 1 \rangle}, \cdots, \mu_i (\DMO)^{\langle p^{\NN -1} \rangle}$ commute with each other.
We refer to $\mu_i (\nabla)$ and  $\mu_i (\DMO)^{\langle \bullet \rangle}$ as the {\bf monodromy operator} of $\nabla$ at $\sigma_i$.
In the case where $\mcF$ is a line bundle, 
$\mu_i (\nabla)^{\langle \bullet \rangle}$ may be regarded as an element of $\Gamma (S, \mcO_S)^{\oplus \NN} \left(= \mr{End}_{\mcO_S} (\sigma_i^* \mcF)^{\oplus \NN} \right)$.

\subsection{Parabolic $p^\NN$-flat bundles} \label{SS81}

In what follows, we generalize the classical definition of a parabolic connection to our situation.

\bde \label{Def9}
Let $\msF := (\mcF, \vec{\mff})$
  be a quasi-parabolic 
   bundle on $\msX$.
A {\bf $\mcD^{(\NN -1)}$-module structure} on $\msF$ is a $\mcD^{(\NN -1)}$-module structure  $\nabla$ on $\mcF$ 
such that $\mr{Ker}(\varpi_{\msF, i}^{[j]})$ is closed under the $\mcD^{(\NN -1)}$-action $\nabla$ for every $i =1, \cdots, r$ and $j = 0, \cdots, \N_i$.
\ede

Let $\msF := (\mcF, \vec{\mff})$ be a quasi-parabolic bundle on $\msX$ and  $(\N_1, \cdots, \N_r)$   the length of $\vec{\mff} = (\mff_1, \cdots, \mff_r)$.
Also, let $\nabla$ be  a $\mcD^{(\NN -1)}$-module structure on $\msF$.
Since 
$\mr{Ker}(\varpi_{\msF, i}^{[j]})$
 (for each $i=1, \cdots, r$ and $j =1, \cdots \N_i$)
 is closed under $\nabla$,
we see that 
$\nabla$ induces a $\mcB_S$-module structure
\begin{align} \label{EQ52}
\mu_i^{[j]} (\nabla) : \mcB_S \rightarrow \mcE nd_{\mcO_S} (\mr{Ker}(\mff_i^{[j]}))
\end{align}
 on
$\mr{Ker}(\mff_i^{[j]}) \left(\cong \mr{Ker}(\varpi_{\msF, i}^{[j-1]})/\mr{Ker}(\varpi_{\msF, i}^{[j]}) \right)$; this  corresponds to 
\begin{align} \label{EQ77}
\mu_i^{[j]} (\nabla)^{\langle \bullet \rangle} := (\mu_i^{[j]} (\nabla)^{\langle 1 \rangle}, \mu_i^{[j]} (\nabla)^{\langle p \rangle}, \cdots, \mu_i^{[j]} (\nabla)^{\langle p^{\NN -1} \rangle}) \in \mr{End}_{\mcO_S}  (\mr{Ker}(\mff_i^{[j]}))^{\oplus \NN}.
\end{align}

Here, for each $\N \in \mbZ_{> 0}$,
we shall denote by
\begin{align} \label{EQ401}
\Xi_{\N, \NN}^{\leq}  \ \left(\text{resp.,} \  \Xi_{\N, \NN}^{<} \right)
\end{align}
the set of 
  $\N$-tuples of  integers $(a^{[1]}, \cdots, a^{[\N]})$ with  $0 \leq a^{[1]} \leq  \cdots \leq a^{[\N]} < p^\NN$ (resp., $0 \leq a^{[1]} < \cdots < a^{[\N]} < p^\NN$).
Also, given   an element $d$ of $\mbZ/p^{\NN}\mbZ$, we  
 denote by $\widetilde{\EX}$ the  integer defined as the unique lifting of $d$ via the natural surjection $\mbZ \twoheadrightarrow \mbZ/p^{\NN}\mbZ$ satisfying $0 \leq \widetilde{\EX} <p^{\NN}$.
Let $\widetilde{\EX}_{[0]}, \cdots, \widetilde{\EX}_{[\NN -1]}$
be 
the  collection of  integers uniquely determined by the condition that    $\widetilde{d}= \sum\limits_{s=0}^{\NN -1}p^s \cdot \widetilde{d}_{[s]}$ and $0 \leq \widetilde{d}_{[s]} <p$ ($s = 0, \cdots, \NN -1$).
For  each  $s= 0, \cdots, \NN -1$, we write
 $\widetilde{\EX}_{[0, s]} := \sum\limits_{s'=0}^{s} p^{s'} \cdot \widetilde{d}_{[s']}$ and write
$d_{[s]}$ (resp., $d_{[0, s]}$) for  the image of $\widetilde{d}_{[s]}$ (resp., $\widetilde{d}_{[0, s]}$) via the natural projection $\mbZ \twoheadrightarrow  \mbF_p$ (resp., $\mbZ \twoheadrightarrow \mbZ/p^{s+1}\mbZ$).

\bde\label{Def66}
\begin{itemize}
\item[(i)]
Let us consider a collection of data
\begin{align}
\msF_\flat := (\mcF, \nabla,  \vec{\mff}, \vec{a})
\end{align}
where 
\begin{itemize}
\item
$(\mcF, \vec{\mff}, \vec{a})$ denotes a parabolic bundle on $\msX$ such that $\vec{a}$ belongs to $\prod_{i=1}^r \Xi_{\N_i, \NN}^{\leq}$ for some $(\N_1, \cdots, \N_r) \in \mbZ_{> 0}^{\times r}$ (in particular,  $(\N_1, \cdots,  \N_r)$ coincides with  the length of $\vec{\mff}$);
\item
$\nabla$ denotes a $\mcD^{(\NN -1)}$-module structure on $(\mcF, \vec{\mff})$ with $\psi_{(\mcF, \nabla)}= 0$ (resp.,  $(\psi_{(\mcF, \nabla)})^{p-1} =0$).
\end{itemize}
Then, we say that $\msF_\flat$ is a {\bf parabolic $p^\NN$-flat} (resp., {\bf parabolic $p^\NN$-nilpotent}) {\bf bundle} on $\msX$
if,    for any  $i=1, \cdots, r$,  $j =1, \cdots, \N_i$,  and $s = 0, \cdots, \NN -1$,
the element $\mu_i^{[j]} (\nabla)^{\langle p^s \rangle} \in \mr{End}_{\mcO_S} (\mr{Ker}(\mff_i^{[j]}))$ coincides with the multiplication by $(-a_i^{[j]})_{[s]} \in \mbF_p$.

The notion of a morphism between parabolic $p^\NN$-flat (resp., parabolic $p^\NN$-nilpotent) bundles can be formulated naturally.
\item[(iii)]
Let $\msF_\flat := (\mcF, \nabla, \vec{\mff}, \vec{a})$ be a parabolic $p^\NN$-flat bundle on $\msX$.
By a {\bf parabolic $p^\NN$-flat subbundle} of $\msF_\flat$,
we mean a parabolic $p^\NN$-flat bundle of the form $\msE_\flat := (\mcE, \nabla |_\mcE, \vec{\mff}_\mcE, \vec{a})$, where $(\mcE, \vec{\mff}_\mcE, \vec{a})$ is a parabolic subbundle of $\msF := (\mcF, \vec{\mff}, \vec{a})$  such that $\mcE$ is closed under the $\mcD^{(\NN -1)}$-action $\nabla$, and $\nabla |_\mcE$ denotes the resulting $\mcD^{(\NN -1)}$-action induced  by $\nabla$. 
\end{itemize}
\ede

Each element  $\vec{a} := (a_1, \cdots, a_r)$ of $\prod_{i=1}^r \Xi_{\N_i, \NN}^{\leq}$ (where $\N_1, \cdots, \N_r \in \mbZ_{>0}$) determines  
 the category
\begin{align} \label{EQ150}
\mcD^{(\NN -1)}\text{-}\mcB un^{\psi =0}_{\msX, \vec{a}} 
\ \left(\text{resp.,} \   \mcD^{(\NN -1)}\text{-}\mcB un^{\mr{nilp}}_{\msX, \vec{a}} \right)
\end{align}
 of  parabolic $p^\NN$-flat 
 (resp., $p^\NN$-nilpotent) 
 bundles  on $\msX$ of (length $(\N_1, \cdots, \N_r)$ and) parabolic weights $\vec{a}$.

\subsection{Exponent of a $p^\NN$-flat bundle} \label{SS49}

Consider the formal disc
 \begin{align} \label{Eq233}
 U := \mcS pec (\mcO_S [\! [t]\!])
 \end{align}
 over $S$, where $t$ denotes a formal parameter; it has a closed immersion $\sigma_\circ : S \hookrightarrow U$ determined by ``$t = 0$".
 The pair 
 \begin{align} \label{EQw34}
 \msU := (U, \sigma_\circ)
 \end{align}
  may be regarded as a ``{\it local curve}" over $S$.  
 In fact, 
 the formal completion of $\msX$ along  the marked point  is, locally on $S$,  isomorphic to $U$.

 Note that various formulations and definitions 
 discussed so far are available when $\msX$ is replaced with  $\msU$.
 In particular, if
  $U^\mr{log}$ denotes the log scheme $U$ equipped with the log structure given by the relative divisor $D_\circ := \mr{Im}(\sigma_\circ)$, then we obtain  $\mcD^{(\NN -1)}_\circ := \mcD^{(\NN -1)}_{U^\mr{log}/S}$.
 The $\mcO_{U}$-module  ${^L}\mcD_{\circ}^{(\M)}$ decomposes into  the direct sum 
$\bigoplus_{j \in \mbZ_{\geq 0}} \mcO_U \cdot \partial^{\langle j \rangle}_\circ$, where 
$\{ \partial^{\langle j \rangle}_\circ \}_j$ is the basis associated to the logarithmic  coordinate  determined by $t$.

For $d \in \mbZ /p^\NN \mbZ$, there exists a unique $\mcD_\circ^{(\NN -1)}$-module structure 
\begin{align} \label{YY52}
\DMO_{\EX}: \mcD_{\circ}^{(\NN -1)} \rightarrow \mcE nd_{\mcO_S} (\mcO_{U})
\end{align}
on $\mcO_{U}$  determined by the condition that
$\DMO_{\EX} (\partial^{\langle j \rangle}_\circ) (t^n) = q_j ! \cdot \binom{n -\widetilde{\EX}}{j}\cdot t^n$
   for every $j$, $n \in \mbZ_{\geq 0}$.
 The resulting $\mcD_{\circ}^{(\NN -1)}$-module 
\begin{align} \label{UU3}
\msO_{d, \flat}
   := (\mcO_{U}, \DMO_{\EX})
\end{align}
 is isomorphic  to the unique extension of 
 $\msO_{0, \flat}$
   to $t^{-\widetilde{\EX}}\cdot \mcO_{U} \left(\supseteq \mcO_U \right)$.
In particular,  $\DMO_{\EX}$ has vanishing $p^{\NN}$-curvature.
Also, it follows from ~\cite[Proposition 4.3.2]{Wak4} that 
the equality 
\begin{align} \label{Eq228}
\mu (\DMO_{\EX})^{\langle \bullet \rangle} = ((-\EX)_{[0]}, \cdots, (-\EX)_{[\NN -1]})
\end{align}
holds under the natural identification $\mr{End}_{\mcO_S} (\sigma_\circ^* \mcO_U) = \Gamma (S, \mcO_S) \left(\supseteq \mbF_p \right)$.

Let $(\mcF, \nabla)$ be a $\mcD_\circ^{(\NN -1)}$-module with vanishing $p^\NN$-curvature such that $\mcF$ is a vector bundle of rank $n > 0$.
Suppose first  that $S = \mr{Spec}(R)$ for a local ring $(R, \mfm)$ over $k$ with $R/\mfm = k$.
According to  ~\cite[Proposition-Definition 4.4.1]{Wak4},  there exists an isomorphism of $\mcD_\circ^{(\NN -1)}$-modules
\begin{align} \label{Eq234}
\bigoplus_{j=1}^n \msO_{d^{[j]}, \flat} \xrightarrow{\sim} (\mcF, \nabla)
\end{align}
for some $d^{[1]}, \cdots, d^{[n]}\in \mbZ/p^\NN \mbZ$;
the multiset $[d^{[1]}, \cdots, d^{[n]}]$ depends only on the isomorphism class of $(\mcF, \nabla)$.
Even in the case where  $S$ is an arbitrary  connected $k$-scheme (that is not necessarily the spectrum of a local ring), 
we can find a well-defined  multiset  of integers
\begin{align} \label{EQ177}
e (\nabla) := [a^{[1]}, \cdots, a^{[n]}]
\end{align}
in $\{0, 1, \cdots, p^\NN-1 \}$ satisfying the following condition:
 if  $d^{[j]}$ denotes   the image of $a^{[j]}$ via $\mbZ \twoheadrightarrow \mbZ /p^\NN \mbZ$ (hence $\widetilde{d}^{[j]} = a^{[j]}$), then
$(\mcF, \nabla)$ admits an isomorphism as in \eqref{Eq234} when restricted over the spectrum of the local ring determined by every closed point of $S$.

\bde\label{Def67}
We refer to  $e (\nabla)$ as the {\bf exponent} of $(\mcF, \nabla)$ (or, of $\nabla$).
(This definition is equivalent to the one discussed in ~\cite[Proposition-Definition 4.4.1]{Wak4} under the identification $d \leftrightarrow \widetilde{d}$.)
\ede

\bde\label{Def78}
Let $(\mcF, \nabla)$ be a $p^\NN$-flat bundle on $\msX$ and $i$  an element of $\{1, \cdots, r\}$.
Also, let $e$ be a multiset of elements in $\{ 0, \cdots, p^\NN -1 \}$ with cardinality $n$.
Then,  we say that
{\bf $(\mcF, \nabla)$  has exponent $e$ at $\sigma_i$}
if, locally on $S$, 
the restriction of $(\mcF, \nabla)$ to the formal completion $\msU'$
of $X$ along  $\mr{Im}(\sigma_i)$
 has  exponent $e$ in the sense of Definition \ref{Def67} under some (and hence, any) identification $\msU' \cong \msU$.
\ede

\section{Frobenius pull-back of parabolic bundles} \label{S28}

In this section,
we introduce the Frobenius pull-back of a parabolic bundle, generalizing the usual definition for  the non-parabolic case.  
The main theorem of this section shows that this operation yields a bijective correspondence between  parabolic bundles on $\msX^{(\NN)}$ and parabolic $p^\NN$-flat bundles on $\msX$ (cf. Theorem \ref{Prop11}).

\subsection{Parabolic Frobenius pull-back in a local situation I} \label{SS31}

As a first step, we  consider the case where the underlying space is the $\NN$-th  Frobenius twist  $\msU^{(\NN)} := (U^{(\NN)}, \sigma_\circ^{(\NN)})$ of $\msU = (U, \sigma_\circ)$ (cf. \eqref{Eq233}) and $S$ is connected.

 Let $\N_\circ$ be a positive integer and 
  $a_\circ := (a_\circ^{[1]}, \cdots, a_\circ^{[\N_\circ]})$ an element of $\Xi_{\N_\circ, \NN}^{\leq}$.
We shall set   
\begin{align} \label{EQR3}
a_\circ /p^\NN := 
(a_\circ^{[1]}/p^\NN, \cdots, a_\circ^{[\N_\circ]}/p^\NN).
\end{align}
Just as in the case of $\msX$,  both the categories $\mcB un_{\msU^{(\NN)},  a_\circ/p^\NN}$ and $\mcD^{(\NN -1)}\text{-}\mcB un_{\msU,  a_\circ}$ 
 for this local situation  can be defined in the same manner.

First, we construct 
an object of  $\mcD^{(\NN -1)}\text{-}\mcB un_{\msU, a_\circ}$ by using 
an object of  $\mcB un_{\msU^{(\NN)}, a_\circ/p^\NN}$.
Let us take a parabolic bundle $\msE := (\mcE, \mfe_\circ, a_\circ/p^\NN)$ classified by 
$\mcB un_{\msU^{(\NN)}, a_\circ/p^\NN}$;
the  quasi-parabolic structure  $\mfe_\circ$ is given by
\begin{align} \label{EQ44}
\mfe_{\circ} : \sigma_\circ^*\mcE  =   \mcE_\circ^{[\N_\circ]} 
\xrightarrow{\mfe_\circ^{[\N_\circ]}}
\mcE_{\circ}^{[\N_\circ-1]} 
\xrightarrow{\mfe_\circ^{[\N_\circ-1]}}
 \cdots 
 \xrightarrow{\mfe_\circ^{[2]}}
 \mcE_\circ^{[1]} 
 \xrightarrow{\mfe_\circ^{[1]}}
  \mcE_\circ^{[0]}=0.
\end{align} 
For simplicity, we write $\mcG^{[j]} := \sigma_{\circ *}\mcE_\circ^{[j]}$ ($j =0, \cdots, m_\circ$).
For  $j \in 1, \cdots, \N_\circ$,
the natural surjection $\varpi_{\msE}^{[j]} :\mcE \twoheadrightarrow \mcG^{[j]}$ induces,  via pull-back along $F^{(\NN)}_{U/S}$, a surjective morphism of $\mcD^{(\NN -1)}$-modules  
$F^{(\NN)*}_{U/S}\varpi_{\msE}^{[j]} : (F^{(\NN)*}_{U/S}\mcE, \nabla_\mcE^\mr{can}) \twoheadrightarrow (F^{(\NN)*}_{U/S}\mcG^{[j]}, \nabla_{\mcG^{[j]}}^\mr{can})$.
Then, we obtain a composite surjection between $\mcD^{(\NN -1)}$-modules
\begin{align}
\widetilde{\varpi}_{\msE}^{[j]}  : & \   
((F^{(\NN)*}_{U/S}\mcE)(p^\NN \cdot D_\circ), \nabla_{\mcE, p^\NN \cdot D_\circ}^\mr{can})  \\ 
&
\twoheadrightarrow
((F^{(\NN)*}_{U/S}\mcG^{[j]})(p^\NN \cdot D_\circ), \nabla_{\mcG^{[j]}, p^\NN \cdot D_\circ}^\mr{can})
\notag \\
&\twoheadrightarrow ((F^{(\NN)*}_{U/S}\mcG^{[j]})(p^\NN \cdot D_\circ), \nabla_{\mcG^{[j]}, p^\NN \cdot D_\circ}^\mr{can}) \otimes  (\mcO_U/ \mcO_U ((a^{[j]}_\circ -p^\NN)\sigma_\circ), \overline{\nabla}_0), \notag
\end{align}
where 
$\overline{\nabla}_{0}$ denotes the $\mcD^{(\NN -1)}$-module structure  on $\mcO_U/ \mcO_U ((a_\circ^{[j]}-p^\NN)\sigma_\circ)$ induced by
$\nabla_0$,
the first arrow is  obtained by extending $F^{(\NN)*}_{U/S}\varpi_{\msE}^{[j]}$, 
 and the second arrow arises from the natural quotient $(\mcO_U, \nabla_0) \twoheadrightarrow (\mcO_U/ \mcO_U ((a_\circ^{[j]}-p^\NN)\sigma_\circ), \overline{\nabla}_0)$. 
Hence, the kernel of this composite determines a $\mcD^{(\NN -1)}$-module 
$(\mr{Ker}(\widetilde{\varpi}_{\msE}^{[j]}), \nabla^{[j]}_\mr{Ker})$.
By applying this construction to various $j$'s and taking their intersection, we obtain a $\mcD^{(\NN -1)}$-module
\begin{align} \label{EQ102}
\msE^F = (\mcE^F, \nabla^F) := \bigcap_{j=1}^{\N_\circ} (\mr{Ker}(\widetilde{\varpi}_{\msE}^{[j]}), \nabla^{[j]}_\mr{Ker}).
\end{align}
Its restriction to $U_\circledast  := U \setminus \mr{Im}(\sigma_\circ)$ may be identified with the usual  Frobenius pull-back of $\mcE$ with the canonical connection.
This implies that $\nabla^F$ has vanishing $p^\NN$-curvature.

Moreover, 
we shall  set
\begin{align} \label{EQ101}
\mcE^{F, [j]}_\circ := 
\sigma^*_\circ\left(\mcE^F/(\mcE^F  \cap \mr{Ker}(F^{(\NN)*}_{U/S}\varpi_{\msE}^{[j]})(-D_\circ)) \right).
\end{align}
Then, there exists a sequence of natural surjections
\begin{align} \label{EQ100}
\mfe^F_{\circ} : \mcE^{F, [\N_\circ]}_\circ \xrightarrow{\mfe_\circ^{F, [\N_\circ]}} \mcE^{F, [\N_\circ-1]}_\circ \xrightarrow{\mfe_\circ^{F, [\N_i-1]}} \cdots  \xrightarrow{\mfe_\circ^{F, [2]}} \mcE^{F, [1]}_\circ
\xrightarrow{\mfe_\circ^{F,  [1]}}
 \mcE^{F, [0]}_\circ = 0.
\end{align}

\bpr \label{Prop10}
  The resulting collection
 \begin{align} \label{EQ111}
 \msE_\flat^F := (\mcE^F, \nabla^F, \mfe^F_\circ, a_\circ)
 \end{align}
 forms a parabolic $p^\NN$-flat bundle on $\msU$ 
 classified by $\mcD^{(\NN -1)}\text{-}\mcB un_{\msU, a_\circ}^{\psi = 0}$,
 and  the equality $\mcS ol (\nabla^F) = \mcE$ holds.
 Moreover, 
 if 
   $\LL_\circ  := (\LL_\circ^{[1]}, \cdots, \LL_\circ^{[\N_\circ]})$ denotes   the type of $\mfe_\circ$,
then
\begin{align} \label{EQ110}
(\mr{det}(\mcE^F), \mr{det}(\nabla^F)) \cong
((F^{(\NN)*}_{U/S}\mr{det}(\mcE))(s D_\circ), \nabla^\mr{can}_{\mr{det}(\mcE), s D_\circ}),
\end{align}
where $s := \sum\limits_{j=1}^{\N_\circ} a_\circ^{[j]} \LL^{[j]}_\circ$.
\epr
\begin{proof}
We shall prove the first assertion.
Let us  set 
 \begin{align} \label{EQ410}
 \mcH := \bigoplus_{j' =1}^{\N_\circ} \mcO_{U^{(\NN)}}^{\oplus \LL_\circ^{[j']}}
 \hspace{5mm}  \text{and} \hspace{5mm}
 \mcH_\circ^{[j]} := \bigoplus_{j'=1}^j\mcO_S^{\oplus \LL_\circ^{[j']}} \  (j=0, \cdots, \N_\circ).
 \end{align}
 In particular, we have $\sigma_\circ^* \mcH = \bigoplus_{j' =1}^{\N_\circ} \mcO_{S}^{\oplus \LL_\circ^{[j']}}$.
 For each $j=1, \cdots, \N_\circ$, denote by $\mfh_\circ^{[j]}$
 the projection $\mcH_\circ^{[j]} \twoheadrightarrow \mcH_\circ^{[j-1]}$ given by forgetting the last direct summand.
 The triple 
 \begin{align} \label{EQ414}
 \msH := (\mcH, \mfh_\circ, a_\circ/p^\NN),
 \end{align}
  where $\mfh_\circ := \{ \mfh_\circ^{[j]} \}_{j=1}^{\N_\circ}$, specifies a parabolic bundle on $\msU^{(\NN)}$.
 Since  the assertion
 is of local nature,
 we can suppose  that there exists an isomorphism  $\tau :  \msE   \xrightarrow{\sim}  \msH$.
 It follows that  the problem can be  reduced to the case of 
$\msE = \msH$.

 For each $j = 1, \cdots, \N_\circ$,
we have 
\begin{align} \label{EQ118}
(\mr{Ker}(\widetilde{\varpi}_{\msH}^{[j]}), \nabla_{\mr{Ker}}^{[j]}) = \bigoplus_{j' =1}^{j} (\mcO_U ( a^{[j']}_\circ  D_\circ), \nabla_0)^{\oplus \LL_\circ^{[j']}} \oplus \bigoplus_{j' = j+1}^{\N_\circ} (\mcO_U ( p^\NN  D_\circ), \nabla_0)^{\oplus \LL^{[j']}_\circ},
\end{align}
where we abuse notation by  writing  $\nabla_0$  for  the restrictions of the trivial $\mcD^{(\NN -1)}$-module structure  on $\mcO_U$ to various $\mcO_X$-submodules.
This implies 
\begin{align} \label{EQ415}
\msH^F = \bigoplus_{j' =1}^{\N_i}   (\mcO_U ( a^{[j']}_\circ  D_\circ), \nabla_0)^{\oplus \LL_{j'}},
\hspace{5mm} 
\mcH_\circ^{F[j]} = \bigoplus_{j' =1}^j
\sigma_\circ^*\left(\mcO_U ( a^{[j']}_\circ  D_\circ)^{\oplus \LL_{j'}}\right)
\end{align}
($j=0, \cdots, \N_\circ$), and the projection  $\mfh_\circ^{F[j]}$ is given by forgetting the last direct summand.
Thus,  $(\mcE^F, \mfh_\circ^F)$ forms a quasi-parabolic bundle on $\msU$.
Moreover, 
since $(\mcO_U (a_\circ^{[j]}D_\oslash), \nabla_0) \cong (\mcO_U, \nabla_{a_\circ^{[j]}})$,
 the description of $\nabla_{\mr{Ker}}^{[j]}$ obtained in \eqref{EQ118} and the equality  \eqref{Eq228}  together imply that
$\msH^F_\flat := (\mcE^F, \nabla^F, \mfh_\circ^F, a_\circ)$ belongs to $\mcD^{(\NN -1)}\text{-}\mcB un_{\msU, a_\circ}$.
This completes the proof of  the first assertion.

The second assertion  follows from the definition of $(\mcE^F, \nabla^F)$ together with the discussion in the proof of the first assertion.
\end{proof}

Note that the formation of $\msE \mapsto \msE^F_\flat$ is functorial.
This is to say, 
if $h : \msE \rightarrow \msE'$ is a morphism in $\mcB un_{\msU^{(\NN)}, a_\circ/p^\NN}$,
then this yields  naturally a morphism $h^F : \msE^F_\flat \rightarrow {\msE'_\flat}^F$.  
Thus, the  assignments $\msE \mapsto \msE_\flat^F$ and $h \mapsto h^F$ together  define a functor
$\mcB un_{\msU^{(\NN)}, a_\circ/p^\NN}\rightarrow \mcD^{(\NN -1)}\text{-}\mcB un_{\msU, a_\circ}^{\psi = 0}$.

\subsection{Parabolic Frobenius pull-back in a local situation  II} \label{SS50}

Next, we give an inverse direction of the above construction.
Let  $\msF_\flat := (\mcF, \nabla, \mff_\circ, a_0)$ (where $\msF := (\mcF, \mff_\circ)$, $\mff_\circ := \{ \mff_\circ^{[j]} \}_{j=1}^{\N_\circ}$) be a parabolic $p^\NN$-flat bundle  on $\msU$ classified by $\mcD^{(\NN -1)}\text{-}\mcB un_{\msU, a_\circ}^{\psi = 0}$.
We shall set $\mcF^\nabla := \mcS ol (\nabla)$, 
which specifies a rank-$n$ vector bundle on $U^{(\NN)}$.
For each $j=0, \cdots, \N_\circ$,
 $\nabla$ restricts to a $\mcD^{(\NN -1)}$-module structure  $\nabla^{[j]}$ on $\mr{Ker} (\varpi_{\msF}^{[j]}) (-(a_\circ^{[j]} + 1) D_\circ)$.
Then, $\mcS ol (\nabla^{[j]})$ specifies a vector bundle on $U^{(\NN)}$ equipped with a natural inclusion $\mcS ol (\nabla^{[j]}) \hookrightarrow \mcS ol (\nabla)$, which becomes an isomorphism when restricted to $U^{(\NN)} \setminus \mr{Im} (\sigma_\circ^{(\NN)})$.
The $\mcO_S$-modules $\mcF^{\nabla [j]}_\circ  := \sigma_\circ^{(\NN)*} (\mcS ol (\nabla)/ \mcS ol (\nabla^{[j]}))$ fit into a sequence of surjections
\begin{align} \label{EQ129}
\mff_{\circ}^\nabla :
\sigma_\circ^{(\NN)*}\mcF^\nabla  =   \mcF_\circ^{\nabla [\N_\circ]} 
\xrightarrow{\mff_\circ^{\nabla[\N_\circ]}}
\mcF_{\circ}^{\nabla [\N_\circ-1]} 
\xrightarrow{\mff_\circ^{\nabla [\N_\circ-1]}}
 \cdots 
 \xrightarrow{\mff_\circ^{\nabla [2]}}
 \mcF_\circ^{\nabla [1]} 
 \xrightarrow{\mff_\circ^{\nabla [1]}}
  \mcF_\circ^{\nabla [0]}=0.
\end{align} 
Then, we can prove the following assertion.

\bpr\label{Prop52}
\begin{itemize}
\item[(i)]
The collection
\begin{align}
\msF^{\nabla}  := (\mcF^\nabla, \mff^\nabla_\circ, a_\circ/p^\NN) 
\end{align}
forms a parabolic bundle on $\msU^{(\NN)}$ classified by $\mcB un_{\msU^{(\NN)}, a_\circ/p^\NN}$.
\item[(ii)]
Let $\msE$
and $\msF_\flat$ be objects in $\mcB un_{\msU^{(\NN)}, a_\circ/p^\NN}$ and $\mcD^{(\NN -1)}\text{-}\mcB un_{\msU, a_\circ}^{\psi = 0}$, respectively.
Then,  there exist canonical isomorphisms
\begin{align} \label{EQ160}
\msE \xrightarrow{\sim} (\msE_\flat^F)^\nabla,
\hspace{5mm} 
(\msF^\nabla)^F_\flat \xrightarrow{\sim} \msF_\flat
\end{align}
that are functorial with respect to $\msE$ and $\msF_\flat$, respectively.
Moreover, when restricted to $U \setminus \mr{Im}(\sigma_\circ)$,  these isomorphisms  coincide with the natural isomorphisms  
$\mcE \xrightarrow{\sim} (F^{(\NN)*}_{U/S}\mcE)^\nabla$ and $F^{(\NN)*}_{U/S}(\mcF^\nabla) \xrightarrow{\sim} \mcF$, respectively.
\end{itemize}
\epr
\begin{proof}
First, we shall prove assertion (i).
Denote by $\LL_\circ := (\LL_\circ^{[1]}, \cdots, \LL_\circ^{[\N_\circ]})$ the type of $\mff_\circ$.
The assertion  is of local nature, 
so  $\msF := (\mcF, \nabla)$ may be identified with a direct sum of $\msO_{d, \flat}$'s (cf. the discussion around \eqref{Eq234}).
Since $(\mcF, \nabla)$ has parabolic  weights $a_0$,
the equality \eqref{Eq228} implies that the problem is reduced to the situation where $\msF$ satisfies the following conditions:
\begin{itemize}
\item
 $\msF$ coincides with  the direct sum
\begin{align} \label{EQ416}
\msF = \bigoplus_{j=1}^{\N_\circ}  (\mcO_U (a_\circ^{[j]} D_\circ), \nabla_0)^{\oplus \LL_\circ^{[j]}}
\end{align}
 where, just as in  the proof of Proposition \ref{Prop10}, we abuse notation by writing 
$\nabla_0$ for various restrictions of the trivial one;
\item
For each $j =0, \cdots, \N_\circ$,  the equality $\mcF_\circ^{[j]} = \bigoplus_{j' = 1}^j \mcO_S^{\oplus \LL_\circ^{[j']}}$ holds.
Also, the projection $\mff_\circ^{[j]}$ coincides with the projection $\mcF_\circ^{[j]}$  given by forgetting the last direct summand.
\end{itemize}

Then, it is verified that 
$\mcF^\nabla = \mcS ol (\nabla) = \bigoplus_{j=1}^{\N_\circ} \mcO_{U^{(\NN)}}^{\oplus \LL_\circ^{[j]}}$  and 
\begin{align} \label{EQ417}
\mcS ol (\nabla^{[j]}) =  \bigoplus_{j' =1}^j \mcO_{U^{(\NN)}} (-D_\circ^{(\NN)})^{\oplus \LL_\circ^{[j']}} \oplus \bigoplus_{j' = j+1}^{\N_\circ} \mcO_{U^{(\NN)}}
\end{align}
 ($j= 0, \cdots, \N_\circ$), where $D^{(\NN)}_\circ := \mr{Im}(\sigma_\circ^{(\NN)})$.
 Hence,  the equality $\mcF^{\nabla [j]}_\circ = \bigoplus_{j' =1}^j \mcO_S^{\oplus \LL_\circ^{[j']}}$ holds and $\mff_\circ^{\nabla [j]} : \mcF^{\nabla [j]}_\circ \twoheadrightarrow \mcF^{\nabla [j -1]}_\circ$ coincides  the projection  given by forgetting the last direct summand.
 Thus, $\msF^\nabla$   specifies  an object in $\mcB un_{\msU^{(\NN)}, a_0/p^\NN}$, and this  completes 
 the proof of assertion (i).

Next, we shall prove assertion (ii).
By considering the various definitions involved, 
one can 
construct the morphisms displayed in \eqref{EQ160}.
To prove that these morphisms are isomorphisms,
 we may assume that $S$ is the spectrum of a local ring.
Thus, 
it suffices to consider the case where $\msE$ and $\msF_\flat$ are of the forms \eqref{EQ414} and \eqref{EQ416}, respectively, and the assertion 
 can be proved immediately from 
the discussions in the proofs of Propositions \ref{Prop10} and \ref{Prop52}, (i). 
\end{proof}

Since the formation of $\msF_\flat \mapsto \msF^\nabla$ is verified to be functorial, 
this operation yields a functor
$\mcD^{(\NN -1)}\text{-}\mcB un_{\msU, a_\circ}^{\psi = 0} \rightarrow \mcB un_{\msU^{(\NN)}, a_\circ/p^\NN}$.

\bt\label{Thm52}
The assignments $\msE \mapsto \msE^F_\flat$ and $\msF_\flat \mapsto \msF^\nabla$ together determine an equivalence of categories
\begin{align} \label{EQ109}
\mcB un_{\msU^{(\NN)}, a_\circ /p^\NN} \xrightarrow{\sim}
\mcD^{(\NN -1)}\text{-}\mcB un_{\msU, a_\circ}^{\psi = 0}.
\end{align}
\et
\begin{proof}
The assertion follows from Proposition
  \ref{Prop52}, (ii).
\end{proof}

\subsection{Parabolic Frobenius pull-back in a global situation} \label{SS89}

As the final step of formulating the Frobenius pull-back,
let us consider the case where underlying space is the globally defined  curve $\msX$.
We may still assume that $S$ is connected.
 
 Let us choose an $r$-tuple of positive integers $(\N_1, \cdots, \N_r)$ and  an element $\vec{a} := (a_1, \cdots, a_r) \in \prod_{i=1}^r \Xi_{\N_i, \NN}^{\leq}$, where $a_i := (a_i^{[1]}, \cdots, a_i^{[\N_i]})$ ($i=1, \cdots, r$).
 We shall set
  \begin{align} \label{EQ420}
  a_i/p^\NN := (a_i^{[1]}/p^\NN, \cdots, a_i^{[\N_i]}/p^\NN), \hspace{8mm}
 \vec{a}/p^\NN := \left(a_1/p^\NN, \cdots, a_r/p^\NN \right).
 \end{align}
 Also, 
 let $\msE := (\mcE, \vec{\mfe}, \vec{a}/p^\NN)$ be a parabolic bundle on $\msX^{(\NN)}$
 belonging to  $\mcB un_{\msX^{(\NN)}, \vec{a}/p^\NN}$.
We apply the 
operation $(-)_\flat^F$ obtained  in \S\,\ref{SS31}
to 
the restriction of $\msE$ to
 the formal neighborhood of  $\sigma_i$ for every $i=1, \cdots, r$.
By gluing together  the resulting parabolic Frobenius pull-backs and the usual Frobenius pull-back over $X \setminus \bigcup_{i=1}^r \mr{Im}(\sigma_i)$, 
we obtain
 a parabolic $p^\NN$-flat bundle
\begin{align} \label{EQ167}
\msE_\flat^F := (\mcE^F, \nabla^F, \vec{\mfe}^F, \vec{a})
\end{align}
 on $\msX$, which belongs to $\mcD^{(\NN -1)}\text{-}\mcB un_{\msX, \vec{a}}^{\psi = 0}$.
If $(\LL_i^{[1]}, \cdots, \LL_i^{[\N_i]})$ denotes the type of $\vec{\mfe}$ at $\sigma_i$, then
 the second assertion in  Proposition \ref{Prop10} implies that
 $\mcS ol (\nabla^F) = \mcE$ and
 \begin{align} \label{EQ21}
 (\mr{det}(\mcE^F), \mr{det}(\nabla^F))
 \cong ((F^{(\NN)*}_{X/S}\mr{det}(\mcE))(D_\star), \nabla^\mr{can}_{\mr{det}(\mcE), D_\star}),
 \end{align}
 where $D_\star := \sum\limits_{i=1}^r \left(\sum\limits_{j=1}^{n_i}a_i^{[j]} \LL_i^{[j]}\right)\sigma_i$.
 This implies
  \begin{align} \label{EQ500}
 \mr{deg} (\mcE^F) = p^\NN \cdot \mr{par}\text{-}\mr{deg} (\msE),
 \hspace{5mm}
 \mu (\mcE^F) = p^\NN \cdot \mr{par}\text{-}\mu (\msE),
 \end{align}
 where $\mr{deg}(-)$ and $\mu (-)$ denote the usual (i.e., non-parabolic) degree and slope, respectively.

On the other hand, let us take a parabolic  $p^\NN$-flat bundle $\msF_\flat$ 
classified by $\mcD^{(\NN -1)}\text{-}\mcB un_{\msX, \vec{a}}^{\psi = 0}$.
Similarly to the above discussion, we can apply the operation  $(-)^\nabla$ obtained  in \S\,\ref{SS50} to 
obtain, from $\msF_\flat$, a parabolic bundle 
\begin{align}
\msF^\nabla := (\mcF^\nabla, \vec{\mff}^{\,\nabla}, \vec{a}/p^\NN)
\end{align}
on $\msX^{(\NN)}$ belonging to $\mcB un_{\msX^{(\NN)}, \vec{a}/p^\NN}$.

\begin{definition} \label{Def113}
 We shall refer to $\msE^F_\flat$  as the {\bf parabolic Frobenius pull-back} of $\msE$ (by $F^{(\NN)}_{X/S}$).
 Also, we shall refer to $\msF^\nabla$ as the {\bf horizontal parabolic bundle} associated to $\msF_\flat$.
\end{definition} 

The following theorem
 can be proved immediately by combining 
the higher-level generalization of Cartier's theorem by B. Le Stum and A. Quir\'{o}s  (cf. ~\cite[Theorem 5.1]{Kal}, ~\cite[Corollary 3.2.4]{LeQu}) and
   Theorem  \ref{Thm52} (or the second assertion of Proposition \ref{Prop52}, (ii)) applied to the formal neighborhoods of the marked points.

 \bt[cf. Theorem \ref{ThC}] \label{Prop11}
 Let $(\N_1, \cdots, \N_r)$ be  an $r$-tuple of positive integers  and $\vec{a} := (a_1, \cdots, a_r)$  an element of $\prod_{i=1}^r \Xi_{\N_i, \NN}^{\leq}$.
Then, the assignments $\msE \mapsto \msE^F_\flat$ and $\msF_\flat \mapsto \msF^\nabla$ together give an equivalence of categories
\begin{align} \label{EQ171}
\mcB un_{\msX^{(\NN -1)}, \vec{a}/p^\NN} \xrightarrow{\sim}
\mcD^{(\NN -1)}\text{-}\mcB un_{\msX, \vec{a}}^{\psi = 0}.
\end{align}
Moreover, the formation of this equivalence is functorial with respect to $S$.
 \et

In particular, we obtain the following assertion.

 \bt \label{Prop26}
 Let us keep the notation in the previous theorem.
 Also, let $\msE$ be a parabolic bundle classified by $\mcB un_{\msX^{(\NN)}, \vec{a}/p^\NN}$.
 Then, for a parabolic subbundle $\msG$ of $\msE$, 
its Frobenius pull-back $\msG^F_\flat$ forms a parabolic $p^\NN$-flat subbundle of $\msE^F_\flat$.
  Moreover,  
  the resulting assignment $\msG \mapsto \msG^F_\flat$
  gives a bijective correspondence between parabolic subbundles of $\msE$ and parabolic $p^\NN$-flat subbundles of $\msE^F_\flat$.
\et
\begin{proof}
The assertions follow from the various definitions involved.
\end{proof}

\subsection{Transitivity of parabolic Frobenius pull-backs} \label{SS94}

Let $M$ be an integer with $0 \leq M \leq \NN$.
In the non-parabolic case, it is immediate that the operation of takings  Frobenius pull-backs is transitive, which means  $F^{(M)*}_{X/S} F^{(\NN -M)*}_{X^{(M)}/S} (-) \cong F^{(\NN)*}_{X/S}(-)$.
This section discusses how can we generalize this transitivity to parabolic Frobenius pull-backs.

For each integer  $a$ in $\{ 0, \cdots, p^\NN -1 \}$,
let $(s_1^M (a), s_2^M (a))$ be  the  unique pair of integers satisfying 
that $a = s_1^M (a) + p^{M} \cdot s_2^M (a)$ and $0 \leq s_1^M (a) <p^M$, $0 \leq s_2^M (a) < p^{\NN -M}$.
We choose
 an $r$-tuple of positive integers $(\N_1, \cdots, \N_r)$ and 
  an element $\vec{a} := (a_1, \cdots, a_r)$, where $a_i = (a_i^{[1]}, \cdots, a_i^{[\N_i]})$,  of $\prod_{i=1}^r \Xi_{\N_i, \NN}^{\leq}$ satisfying
    \begin{align} \label{EQ503}
    s_1^M (a_i^{[1]}) \leq  \cdots \leq s_1^M (a_i^{[\N_i]}) 
    \hspace{5mm} \text{and} \hspace{5mm}  
   s_2^M (a_i^{[1]}) \leq  \cdots \leq  s_2^M (a_i^{[\N_i]}) 
   \end{align}
 for every $i=1, \cdots, r$.
 By putting $s_t^M (a_i) := (s_t^M (a_i^{[1]}), \cdots, s_t^M (a_i^{[\N_i]}))$ ($t \in \{1,2\}$, $i \in \{1, \cdots, r\}$), we obtain a collection of  parabolic weights 
 $s_t^M (\vec{a}) := (s_t^M (a_1), \cdots, s_t^M (a_r))$.

Now, let 
$\msE := (\mcE, \vec{\mfe}, \vec{a}/p^\NN)$ be  a parabolic bundle classified by $\mcB un_{\msX^{(\NN)}, \vec{a}/p^\NN}$; it associates the
parabolic Frobenius pull-back 
 $\mcE^F := (\mcE^F, \nabla^F, \vec{\mfe}^F, \vec{a})$ by 
 $F_{X/S}^{(\NN)}$.
 Also,
 since $s_2^M (\vec{a})$ belongs to $\prod_{i=1}^r \Xi_{\N_i, \NN -M}^{\leq}$, 
  one obtains the parabolic Frobenius pull-back of $(\mcE, \vec{\mff}, s_2^M (\vec{a})/p^{\NN -M})$ by $F_{X^{(M)}/S}^{(\NN - M)} : X^{(M)} \rightarrow X^{(\NN)}$,
 which we  denote by $\mcE' := (\mcE', \nabla', \vec{\mfe}\,', s_2^M (\vec{a}))$.
The resulting triple $(\mcE', \vec{\mfe}\,', s_1^M (\vec{a})/p^M)$ specifies an object in the category  $\mcB un_{\msX^{(M)}, s_1^M (\vec{a})/p^M}$, so it associates the parabolic Frobenius pull-back by $F_{X/S}^{(M)}$; we denote it by
$\msE'' := (\mcE'', \nabla'', \vec{\mfe}\,'', s_1^M (\vec{a}))$.

It follows from ~\cite[Corollaire 3.3.1]{Mon} that
the Frobenius pull-back of $\nabla'$ (in the usual sense) gives a $\mcD^{(\NN -1)}$-module structure on $F^{(M)*}_{X/S}\mcE'$, and the definition of parabolic Frobenius pull-back implies that this extends to a $\mcD^{(\NN -1)}$-module structure $\nabla^{F''}$ on $\mcE''$.
By construction, the $\mcD^{(M-1)}$-module structure on $\mcE''$ obtained from $\nabla^{F''}$ by reducing the level coincides with $\nabla''$.
Moreover,  the following proposition can be proved by the various definitions involved.

\bpr \label{Prop100}
The natural isomorphism $F^{(M)*}_{X/S} F_{X/S}^{(\NN -M)*} \mcE \xrightarrow{\sim} F^{(\NN)*}_{X/S}\mcE$ extends to an isomorphism of  parabolic $p^\NN$-flat bundles
\begin{align} \label{EQ504}
(\mcE'', \nabla^{F''}, \vec{\mfe}\,'', \vec{a}) \xrightarrow{\sim} (\mcE^F, \nabla^F, \vec{\mfe}, \vec{a}).
\end{align}
\epr

\subsection{Parabolic  Higgs bundles}
 \label{SS71}

The 
celebrated
 Ogus-Vologodsky correspondence (cf. ~\cite{OgVo}) generalizes  Cartier's theorem in such a  way that 
 involves   nilpotent Higgs bundles, and there exist some variants including  higher-level $\mcD$-modules (cf. ~\cite{GLQ},  ~\cite{Ohk}, ~\cite{Sch}).
In the rest of this section, we will give another slight extension, i.e., 
  the Ogus-Vologodsky correspondence for parabolic Higgs bundles, extending the equivalence  asserted in Theorem \ref{Prop11}.

Suppose here that we are given a lifting $\widetilde{S}$ of $S$ modulo $p^2$ and a lifting $\widetilde{\msX}^{(\NN)} := (\widetilde{X}^{(\NN)}, \{ \widetilde{\sigma}_i \}_{i=1}^r)$ of  $\msX^{(\NN)}$ over $\widetilde{S}$.
We write $\msX_\emptyset$ for the unpointed smooth curve $(X, \emptyset)$.
Recall from ~\cite{Ohk} that the lifting $\widetilde{X}/\widetilde{S}$ associates 
a canonical equivalence of categories between the following categories:
\begin{itemize}
\item[(a)]
the category of $\mcO_{X^{(\NN)}}$-modules $\mcE$ equipped with a Higgs field $\theta : \mcE \rightarrow \Omega_{X^{(\NN)}/S} \otimes \mcE$ with $\theta^{p -1} = 0$;
\item[(b)]
the category of $\mcD_{X/S}^{(\NN -1)}$-modules $(\mcF, \nabla)$ whose $p^\NN$-curvature $\psi_{(\mcF, \nabla)}$ satisfies $\psi_{(\mcF, \nabla)}^{p-1} = 0$.
\end{itemize}
Note that this equivalence can be formulated in the case where $\msX$ is replaced with the local curve $\msU$ defined  as in \eqref{EQw34} (after fixing a lifting of $\msU^{(\NN)}$ modulo $p^2$).
For each object $(\mcE, \theta)$ (resp., $(\mcF, \nabla)$) in the category  (a) (resp., (b)), the counterpart  via the above equivalence will be denoted by 
 $C^{-1}_{\widetilde{\msX}^{(\NN)}}(\mcE, \theta)$ (resp., $C_{\widetilde{\msX}^{(\NN)}} (\mcF, \nabla)$).

\bde \label{Def4811}
\begin{itemize}
\item[(i)]
A {\bf non-logarithmic parabolic Higgs bundle} on $\msX$ is a collection of data 
\begin{align} \label{EQ330}
\msF_\sharp := (\mcF, \vec{\mff}, \vec{\alpha}, \theta)
\end{align}
consisting of  a parabolic bundle $(\mcF, \vec{\mff}, \vec{\alpha})$   on $\msX$  (as in the resp'd portion of \eqref{Eq201}) and 
a Higgs field
 $\theta : \mcF \rightarrow \Omega_{X/S} \otimes \mcF$
   such that $\mr{Ker}(\varpi^{[j]}_{\msF, i})$ (where $\msF := (\mcF, \vec{\mff})$) is closed under $\theta$ for every $i=1, \cdots, r$ and $j  =  1, \cdots, \M_i$.

Also, one can define the notion of a morphism between two non-logarithmic parabolic Higgs bundles.
\item[(ii)]
We say that a non-logarithmic parabolic Higgs bundle $\msF_\sharp := (\mcF, \vec{\mff}, \vec{\alpha}, \theta)$ is {\bf $p^\NN$-nilpotent} if $\theta^{p-1} = 0$.
\end{itemize}
\ede

Let $(\M_1, \cdots \M_r)$ be an $r$-tuple of positive integers and $\vec{a} := (a_1, \cdots, a_r)$ an element of $\prod_{i=1}^r \Xi_{\M_i, \NN}^{\leq}$.
Denote by
\begin{align} \label{EQR21}
\mcH ig_{\msX, \vec{a}/p^\NN}^{\mr{non}\text{-}\mr{log}}
\end{align}
the category of  $p^\NN$-nilpotent non-logarithmic parabolic Higgs bundles on $\msX$ of parabolic weights $\vec{a}/p^\NN$.
Since each
vector bundle may be identified with a Higgs bundle  with vanishing Higgs field,
$\mcH ig_{\msX, \vec{a}/p^\NN}^{\mr{non}\text{-}\mr{log}}$ contains $\mcB un_{\msX, \vec{a}/p^\NN}$ as a full subcategory.
The notion
introduced in Definition \ref{Def4811}
can be formulated when the underlying curve is replaced with  $\msU$.
In particular, 
we have the category $\mcH ig_{\msU^{(\NN)}, a_\circ/p^\NN}^{\mr{non}\text{-}\mr{log}}$ for parabolic weights $a_\circ/p^\NN$.

In what follows,
we consider the  local construction of $p^\NN$-nilpotent bundles using 
each object in $\mcH ig_{\msU^{(\NN)}, a_\circ/p^\NN}^{\mr{non}\text{-}\mr{log}}$ 
  in the way similar to the argument in \S\,\ref{SS31}.
  Let us choose   a lifting $\widetilde{\msU}^{(\NN)} := (\widetilde{U}^{(\NN)}/\widetilde{S}, \widetilde{\sigma}_\circ)$ of $\msU^{(\NN)}$ modulo $p^2$, and 
 take an object
 $\msE_\sharp := (\mcE, \mfe_\circ, a_\circ/p^\NN, \theta)$ of 
 $\mcH ig_{\msU^{(\NN)}, a_\circ/p^\NN}^{\mr{non}\text{-}\mr{log}}$,
where $\mfe_\circ$ and $a_\circ /p^\NN$ are given by \eqref{EQ44} and \eqref{EQR3}, respectively. 
We write $\mcG^{[j]} := \sigma_{\circ *}\mcE_\circ^{[j]}$ ($j=0, \cdots, \M_\circ$) and
$\msE := (\mcE, \mfe_\circ, a_\circ/p^\NN)$.
Note that $\theta$ preserves $\mr{Ker}(\varpi_\msE^{[j]}) \left( \subseteq \mcE\right)$
for each  $j$,  so 
it induces an $\mcO_U$-linear morphism $\overline{\theta}^{[j]} : \mcG^{[j]}\rightarrow \Omega_{U/S} \otimes \mcG^{[j]}$.
By applying the functor $C_{\widetilde{\msU}^{(\NN)}}^{-1} (-)$ to $(\mcE, \theta)$ (resp., $( \mcG^{[j]}, \overline{\theta}^{[j]})$),
we obtain a vector bundle $\mcH$ on $U$ (resp., an $\mcO_U$-module  $\mcH^{[j]}$) equipped  with a structure of    $\mcD^{(\NN -1)}_{\circ}$-module  $\nabla_\mcH$ (resp., $\nabla_{\mcH}^{[j]}$).
Since $\widetilde{\sigma}_\circ$ defines a marked point on $\widetilde{U}^{(\NN)}$, 
 the construction of the equivalence
between the categories  (a) and  (b) implies that, 
 for every $i=1, \cdots, r$,
$\mcH^{[j]}$ is supported on $\mr{Im}(\sigma_\circ)$ and obtained as the direct image of a vector bundle on $S$ via $\sigma_\circ$.
The surjection $\varpi_{\msE}^{[j]} : \mcE \twoheadrightarrow \mcG^{[j]}$ yields a surjection 
$\check{\varpi}_{\msE}^{[j]} : (\mcH, \nabla_\mcH) \twoheadrightarrow (\mcH^{[j]}, \nabla_{\mcH}^{[j]})$.
Consider  the following composite surjection of $\mcD^{(\NN -1)}_\circ$-modules:
\begin{align}
\widetilde{\varpi}^{[j]}_{\msE} & :  (\mcH (p^\NN \cdot D_\circ), \nabla_{\mcH, p^\NN \cdot D_\circ}) \\
& \ \  \twoheadrightarrow (\mcH_i^{[j]} (p^\NN \cdot D_\circ), \nabla_{\mcH,  p^\NN \cdot D_\circ}^{[j]}) \notag \\
& \ \  \twoheadrightarrow
(\mcH_i^{[j]} (p^\NN \cdot D_\circ), \nabla_{\mcH, p^\NN \cdot D_\circ}^{[j]}) \otimes  (\mcO_X/ \mcO_X ((a_\circ^{[j]}-p^\NN) \overline{\sigma}_i), \overline{\nabla}_0)
\end{align}
where
\begin{itemize}
\item
$\nabla_{\mcH, p^\NN \cdot D_\circ}$ (resp., $\nabla^{[j]}_{\mcH, p^\NN \cdot D_\circ}$)  denotes the $\mcD^{(\NN -1)}_\circ$-module structure on $\mcH (p^\NN \cdot D_\circ)$ (resp., $\mcH^{[j]} (p^\NN \cdot D_\circ)$)  extending $\nabla_\mcH$ (resp., $\nabla^{[j]}_{\mcH}$);
\item
the first arrow arrow is obtained by extending $\check{\varpi}_{\msE}^{[j]}$, and 
the second  arrow arises from the natural quotient $(\mcO_U,  \nabla_0) \twoheadrightarrow (\mcO_U/ \mcO_U ((a_\circ^{[j]}-p^\NN) \overline{\sigma}_\circ), \overline{\nabla}_0)$.
\end{itemize}
The kernel of this composite determines a $\mcD^{(\NN -1)}_\circ$-module
$(\mr{Ker}(\widetilde{\varpi}_{\msE}^{[j]}), \nabla_{\mr{Ker}}^{[j]})$.
Applying this construction to various  $j$'s and then taking their intersection, we obtain a $\mcD_\circ^{(\NN -1)}$-module
\begin{align}
\msE^F = (\mcE^F, \nabla^F) := \bigcap_{j=1}^{\M_\circ} (\mr{Ker}(\widetilde{\varpi}_{\msE}^{[j]}), \nabla_{\mr{Ker}}^{[j]}).
\end{align}
Moreover, 
just as in \eqref{EQ101}  and \eqref{EQ100}, we obtain a quasi-parabolic structure $\mfe^F_\circ$ on $\mcE^F$.
The resulting collection 
\begin{align}
\msE^F_\flat := (\mcE^F, \nabla^F, \mfe_\circ, a_\circ)
\end{align}
forms a parabolic $p^\NN$-nilpotent bundle on $\msU$.
(This construction for vanishing Higgs fields coincides with that in \eqref{EQ167}.)

\bde \label{Def556}
\begin{itemize}
\item[(i)]
Let $\msF_\flat$ be a parabolic $p^\NN$-nilpotent bundle on $\msU$ of parabolic weights $a_\circ$.
We say that $\msF_\flat$  is {\bf of restricted type}
if it is isomorphic to $\msE^F_\flat$ for some object $\msE$ in $\mcH ig_{\msU^{(\NN)}, a_\circ /p^\NN}^{\mr{non}\text{-}\mr{log}}$.
\item[(ii)]
Let $\msF_\flat$
  be a parabolic $p^\NN$-nilpotent bundle on $\msX$.
We say that $\msF_\flat$ is {\bf of restricted type}
if, for each $i=1, \cdots, r$, its restriction to the formal completion  $\msU'$  of $X$ along $\mr{Im}(\sigma_i)$   is, locally on $S$,  of restricted type under some (and hence, any)  identification $\msU' \cong \msU$.
\end{itemize}
\ede

Denote by 
\begin{align} \label{EQR23}
\mcD^{(\NN -1)}\text{-}\mcB un_{\msX, \vec{a}}^{\mr{res}} \ \left(\text{resp.,} \  \mcD^{(\NN -1)}\text{-}\mcB un_{\msU, a_\circ}^{\mr{res}}\right)
\end{align}
the full subcategory of  $\mcD^{(\NN -1)}\text{-}\mcB un_{\msX, \vec{a}}^{\mr{nilp}}$ (resp., $\mcD^{(\NN -1)}\text{-}\mcB un_{\msU, a_\circ}^{\mr{nilp}}$) consisting of 
 parabolic $p^\NN$-nilpotent bundles of restricted type;
 it contains $\mcD^{(\NN -1)}\text{-}\mcB un_{\msX, \vec{a}}^{\psi = 0}$ (resp., $\mcD^{(\NN -1)}\text{-}\mcB un_{\msU, a_\circ}^{\psi = 0}$) as a full subcategory.

\LSP
\subsection{Parabolic generalization of O-V correspondence} \label{SS74}

Let $\msE_\sharp$ and $\msE^F_\flat$ be as in the previous subsection.
Denote by $\mcF$ (resp., $\mcF^{[j]}$)
 the maximal  element (with respect to the inclusion relation) among all the $\mcD_{\circ}^{(\NN -1)}$-submodules $\mcF'$ of $\mcE^F$ (resp., $\mr{Ker}(\varpi^{[j]}_{\msE^F})(-(a_\circ^{[j]}+1)D_\circ)$) 
 whose 
 $\mcD_{\circ}^{(\NN -1)}$-actions  
 extend to $\mcD_{U/S}^{(\NN -1)}$-actions.
By definition, $\mcF$ (resp., $\mcF^{[j]}$) is equipped with 
a $\mcD_{U/S}^{(\NN -1)}$-module structure
 $\nabla_{\mcF}$ (resp., $\nabla_{\mcF^{[j]}}$). 
 By applying the functor  $C_{\widetilde{\msU}^{(\NN)}} (-)$ to $(\mcF, \nabla_{\mcF})$ (resp., $(\mcF^{[j]}, \nabla_{\mcF^{[j]}})$),
 we obtain 
 a Higgs bundle $(\mcE_\star, \theta_\star)$ (resp., $(\mcE_\star^{[j]}, \theta_\star^{[j]})$).
 The $\mcO_S$-module $\mcE_{\star, \circ}^{[j]} := \sigma_\circ^{(\NN)*}(\mcE_\star / \mcE_\star^{[j]})$ fit into a sequence of surjections
 \begin{align}
 \mfe_{\star, \circ} : \sigma_\circ^{(\NN)*} \mcE_\star = \mcE_{\star, \circ}^{[\M_\circ]} \xrightarrow{\mfe_{\star, \circ}^{[\M_\circ]}}
 \mcE_{\star, \circ}^{[\M_\circ -1]}
 \xrightarrow{\mfe_{\star, \circ}^{[\M_\circ -1]}}
 \cdots
 \xrightarrow{\mfe_{\star, \circ}^{[2]}}
 \mcE_{\star, \circ}^{[1]}
 \xrightarrow{\mfe_{\star, \circ}^{[0]}} \mcE_{\star, \circ}^{[0]} = 0.
 \end{align}
 Thus, we have obtained a collection 
  \begin{align}
 \msE_{\star, \sharp} := (\mcE_\star, \theta_\star, \mfe_{\star, \circ}, a_\circ /p^\NN).
 \end{align}

 \ble\label{Def5029}
 $\msE_{\star, \sharp}$ forms a non-logarithmic $p^\NN$-nilpotent   Higgs bundle, and 
 there exists an isomorphism 
$\msE_\sharp \xrightarrow{\sim} \msE_{\star, \sharp}$ functorial with respect to $\msE_\sharp$.
In particular, $\msE_\sharp \mapsto \msE^F_\flat$ defines an equivalence of categories
\begin{align} \label{EQR10}
\mcH ig_{\msU^{(\NN)}, a_\circ /p^\NN}^{\mr{non}\text{-}\mr{log}} \xrightarrow{\sim} \mcD^{(\NN -1)}\text{-}\mcB un_{\msU, a_\circ}^{\mr{res}}.
\end{align}
 \ele
\begin{proof}
It follows from the definition of ``$\mcF$" that there  exists  an inclusion $(\mcH, \nabla_\mcH) \hookrightarrow (\mcF, \nabla_\mcF)$ (where ``$(\mcH, \nabla_\mcH)$" is the $\mcD_\circ^{(\NN -1)}$-module  introduced in the previous subsection), and it  is verified to be an isomorphism.
Thus, this yields, via $C_{\widetilde{\msU}^{(\NN)}}$, an isomorphism $(\mcE, \theta)  \xrightarrow{\sim}   (\mcE_\star, \theta_\star)$, which specifies  an isomorphism $\msE_\sharp \xrightarrow{\sim} \msE_{\star, \sharp}$.
\end{proof}

Now, let us go back to the global situation.
One can glue together the equivalence between the categories (a) and (b) (mentioned above)
 with $X$ replaced by $X \setminus \bigcup_{i=1}^r \mr{Im}(\sigma_i)$ and
the equivalence \eqref{EQR10} in the case where $\msU$ is taken to be the   formal  completion
along every marked point. 
As a result, we obtain  the following assertion.

\bt\label{TT495}
There exists a canonical equivalence of categories
\begin{align} \label{EQ3335}
\mcH ig_{\msX^{(\NN)}, \vec{a}/p^\NN}^\mr{non\text{-}log} \xrightarrow{\sim}\mcD^{(\NN -1)}\text{-}\mcB un_{\msX, \vec{a}}^{\mr{res}}
\end{align}
such that the following square diagram  is $1$-commutative:
\begin{align} \label{Eqwe4}
\vcenter{\xymatrix@C=46pt@R=36pt{
\mcB un_{\msX^{(\NN)}, \vec{a}/p^\NN} \ar[r]^-{\eqref{EQ171}}_-{\sim}  \ar[d]_-{\mr{inclusion}}& 
\mcD^{(\NN -1)}\text{-} \mcB un_{\msX, \vec{a}}^{\psi = 0}
 \ar[d]^-{\mr{inclusion}}
\\
\mcH ig_{\msX^{(\NN)}, \vec{a}/p^\NN}^{\mr{non}\text{-}\mr{log}}\ar[r]^-{\sim}_-{\eqref{EQ3335}} & \mcD^{(\NN -1)}\text{-} \mcB un_{\msX, \vec{a}}^{\mr{res}}.
 }}
\end{align}
Moreover, the formation of this equivalence is functorial with respect to $S$.
\et

\section{Dormant $\mr{GL}_n^{(\NN)}$-opers} \label{S23}

In this section, we
recall from  ~\cite{Wak4} the definition of a dormant $\mr{GL}_n^{(\NN)}$-oper.
Roughly speaking, a dormant $\mr{GL}_n^{(\NN)}$-oper is a $p^\NN$-flat bundle equipped with a complete flag satisfying a strict form of Griffiths transversality.
We will prove in the next section that these objects defined on $\msX$ correspond bijectively to certain parabolic bundles on $\msX^{(\NN)}$ destabilized by  parabolic Frobenius pull-back.

In the rest of the present paper, we assume that the pair of nonnegative integers $(g, r)$ satisfies  $2g-2 +r >0$.

\subsection{Definition of a dormant $\mr{GL}_n^{(\NN)}$-oper} \label{SS311}

Let $n$ be a positive  integer. 
Consider   a collection of  data 
 \begin{align} 
 \label{GL1}
 \msF^\heartsuit := (\mcF, \nabla, \{ \mcF^j \}_{j=0}^n),
 \end{align}
where
\begin{itemize}
\item
$\mcF$ is a vector bundle  on $X$ of rank $n$;
\item
$\nabla$ is a $\mcD^{(\NN -1)}$-module structure on $\mcF$;
\item
$\{ \mcF^j \}_{j=0}^n$ is   an $n$-step decreasing filtration
\begin{equation}
0 = \mcF^n \subseteq \mcF^{n-1} \subseteq \dotsm \subseteq  \mcF^0= \mcF
\end{equation}
 on $\mcF$ consisting of subbundles such that the subquotients  $\mcF^j / \mcF^{j+1}$ are line bundles.
\end{itemize}

\bde \label{y0108} 
\begin{itemize}
 \item[(i)]
We shall say that $\msF^\heartsuit$ is  a {\bf $\mr{GL}_n^{(\NN)}$-oper}   (or a {\bf $\mr{GL}_n$-oper of level $\NN$})
 on  $\msX$ if, for every $j=0, \cdots, n-1$, 
 the $\mcO_X$-linear morphism
  $\mcD^{(\NN -1)} \otimes \mcF \rightarrow \mcF$ induced by $\nabla$ restricts to an isomorphism
  \begin{align} \label{Eq59}
 \mcD_{< n-j}^{(\NN -1)} \otimes \mcF^{n-1} \xrightarrow{\sim} \mcF^j.
 \end{align}
 The notion of an isomorphism between $\mr{GL}_n^{(\NN)}$-opers can be defined in a natural manner.
\item[(ii)]
We shall say that a $\mr{GL}_n^{(\NN)}$-oper $\msF^\heartsuit := (\mcF, \nabla, \{ \mcF^j \}_j)$ is {\bf dormant} if $\nabla$ has vanishing $p^\NN$-curvature.
\end{itemize}  
  \ede

Let  
$\vec{a} := (a_1, \cdots, a_r)$ (where $a_i := (a_i^{[1]}, \cdots, a_i^{[n]})$) be an element of $(\Xi_{n, \NN}^{\leq})^{\times r}$, and let
$\mcL_\flat := (\mcL, \nabla_\mcL)$ be a $p^\NN$-flat line bundle on $\msX$, i.e., a pair consisting of a line bundle $\mcL$ on $X$ and a $\mcD^{(\NN -1)}$-module structure $\nabla_\mcL$ on $\mcL$ with vanishing $p^\NN$-curvature.
For each $S$-scheme $s : S' \rightarrow S$,  we shall write
\begin{align}
\mcO p^{^\mr{Zzz...}}_{\msX, n, \vec{a}, \mcL_\flat} (S')
\end{align}
for the groupoid classifying collections
$(\msF^\heartsuit, \eta)$,
where
\begin{itemize}
\item
 $\msF^\heartsuit := (\mcF, \nabla, \{ \mcF^j \}_{j=0}^n)$ is a dormant $\mr{GL}^{(\NN)}_n$-oper on $\msX_{S'}$ whose exponent at $\sigma_i$ coincides with the multiset $[a^{[1]}_i, \cdots, a^{[n]}_i ]$ for every $i=1, \cdots, r$; 
 \item 
$\eta$ denotes an isomorphism of $\mcD^{(\NN -1)}_{X^\mr{log}_{S'}/S'}$-modules  $\mr{det}(\mcF, \nabla) \xrightarrow{\sim} s^*\mcL_\flat$, where $s^*\mcL_\flat$ denotes the base-change of $\mcL_\flat$ along $s$.
\end{itemize}

Then, we obtain the category fibered in groupoids
\begin{align} \label{Eq405}
\mcO p^{^\mr{Zzz...}}_{\msX, n, \vec{a}, \mcL_\flat}
\end{align}
 over $\mcS ch_{/S}$ whose fiber over  each  $S$-scheme $S'$ is  
$\mcO p^{^\mr{Zzz...}}_{\msX, n, \vec{a}, \mcL_\flat} (S')$.
It is immediately verified that $\mcO p^{^\mr{Zzz...}}_{\msX, n, \vec{a}, \mcL_\flat}$ forms  a (possibly empty)  stack over $S$.
Since the determinants of $\mr{GL}_n^{(\NN)}$-opers in $\mcO p^{^\mr{Zzz...}}_{\msX, n, \vec{a}, \mcL_\flat}$  are isomorphic to $\mcL_\flat$,   this stack becomes empty unless the exponent  of $\mcL_\flat$ at $\sigma_i$  coincides with the sum $\sum\limits_{j=1}^{n} a_i^{[j]}$ modulo $p^\NN$ (for every $i =1, \cdots, r$).
Also,  it follows from ~\cite[Proposition 6.3.5]{Wak4} that this stack becomes empty unless $a_i$ lies in $\Xi_{n, \NN}^{<} \left(\subseteq \Xi_{n, \NN}^{\leq} \right)$ for every $i$.

\begin{rem}\label{EQ201}
Suppose that we are given an object $(\msF^\heartsuit, \eta)$  of $\mcO p^{^\mr{Zzz...}}_{\msX, n, \vec{a}, \mcL_\flat}$, where $\msF^\heartsuit := (\mcF, \nabla, \{ \mcF^j \}_{j=0}^n)$.
 Since \eqref{Eq59} is an isomorphism, there exists  a composite
  isomorphism  
  \begin{align} \label{EQ491}
    \mcF^j/\mcF^{j+1}
   &   \xrightarrow{\sim} (\mcD_{< n-j}^{(\NN -1)} \otimes \mcF^{n-1})/
   (\mcD_{< n-j-1}^{(\NN -1)} \otimes \mcF^{n-1})
    \\
   & 
 \xrightarrow{\sim} (\mcD_{< n-j}^{(\NN -1)}/\mcD_{< n-j-1}^{(\NN -1)}) \otimes \mcF^{n-1}
  \notag \\
   & \xrightarrow{\sim}
   \mcT^{\otimes (n-j-1)} \otimes \mcF^{n-1}
   \notag 
  \end{align}
  for every $j=0, \cdots, n-1$.
 Hence,  we obtain a composite of canonical isomorphisms
 \begin{align} \label{EQ283}
 \mr{det}(\mcF) \xrightarrow{\sim} \bigotimes_{j=0}^{n-1}\mcF^j /\mcF^{j+1}
 \xrightarrow{\sim}  \bigotimes_{j=0}^{n-1} \mcT^{\otimes (n-j-1)} \otimes \mcF^{n-1}
 \xrightarrow{\sim}   \mcT^{\otimes \frac{n (n-1)}{2}} \otimes (\mcF^{n-1})^{\otimes n}.
 \end{align}
In particular, the relative degree  $\mr{deg}(\mcL)$ of $\mcL$ can be calculated as follows:
 \begin{align} \label{EQ456}
 \mr{deg} (\mcL) 
  &= \mr{deg}( \mcT^{\otimes \frac{n (n-1)}{2}} \otimes (\mcF^{n-1})^{\otimes n})  \\
 & =- \frac{n (n-1) (2g-2 +r)}{2} + n \cdot \mr{deg}(\mcF^{n-1}). \notag
 \end{align}
 It follows that $\mcO p^{^\mr{Zzz...}}_{\msX, n, \vec{a}, \mcL_\flat}$ becomes empty  when 
 $n$ does not divide  the integer $\mr{deg}(\mcL) +\frac{n (n-1) (2g-2 +r)}{2}$.
\end{rem}

\subsection{Relation with dormant $\mr{PGL}_n^{(\NN)}$-opers} \label{SS45}

This subsection describes the relationship between   $\mcO p^{^\mr{Zzz...}}_{\msX, n , \vec{a}, \mcL_\flat}$ and the moduli stack of dormant $\mr{PGL}_n^{(\NN)}$-opers.
The precise definition of  a dormant $\mr{PGL}_n^{(\NN)}$-oper (of prescribed radii) is omitted here, so we refer the reader to ~\cite[\S\,5]{Wak4} for details.
Note that we will only use the ``$\NN =1$" case in the subsequent discussion (cf. Theorem \ref{Th22}, (i)),  and that case is comprehensively investigated in ~\cite{Wak2}.

Hereinafter, we assume that $n < p$.
We shall denote by 
\begin{align} \label{EQ199}
\overline{\Xi}_{n, \NN}^{<}
\end{align}
the set of equivalence classes of elements in $\Xi_{n, \NN}^{<}$, where the equivalence relation  is defined as follows: $(a^{[1]}, \cdots, a^{[n]})$ and $(b^{[1]}, \cdots, b^{[n]})$ are {\it equivalent} if there exists an element $c \in \mbZ$ such  that the two subsets of $\mbZ/p^\NN \mbZ$ defined, via the quotient $\mbZ \twoheadrightarrow \mbZ/p^\NN \mbZ$, by the sets $\{ a^{[j]} + c \}_{j=1}^n$ and $\{ b^{[j]} \}_{j=1}^n$ coincide.
(Note that $\overline{\Xi}_{n, \NN}^{<}$ may be regarded as  a subset of ``$\Xi_{n, N}$"  introduced in ~\cite[\S\,6.3]{Wak4}.)
This set has  a natural projection
\begin{align} \label{EQ200}
\rho : \Xi_{n, \NN}^{<} \twoheadrightarrow \overline{\Xi}_{n, \NN}^{<}.
\end{align}
The assumption $n < p$ implies  that, {\it for each $i \in \{1, \cdots , r \}$ and  $d := (d^{[1]}, \cdots, d^{[n]}) \in \overline{\Xi}_{n, \NN}^{<}$, there exists a unique  $(b^{[1]}, \cdots, b^{[n]}) \in \rho^{-1} (d)$ such that $\sum\limits_{j=1}^{n} b^{[j]}$ coincides with the exponent of $\mcL_\flat$ at $\sigma_i$.}

Let $\vec{a} := (a_1, \cdots, a_r)$ be as before, and
write $\rho (\vec{a}) := (\rho (a_1), \cdots, \rho (a_r))$.
According to ~\cite[\S\,6.3]{Wak4},
one can obtain the moduli stack
\begin{align} \label{Eq465}
\overline{\mcO} p^{^\mr{Zzz...}}_{\msX, n, \rho (\vec{a})}
\end{align}
classifying dormant $\mr{PGL}_n^{(\NN)}$-opers on $\msX$ of radii $\rho (\vec{a})$.
 (That is to say, $\overline{\mcO} p^{^\mr{Zzz...}}_{\msX, n, \rho (\vec{a})}$ is  the fiber product of the projection  ``$\Pi_{\rho (\vec{a}), g, r} : \mcO p^{^\mr{Zzz...}}_{\rho (\vec{a}), g, r} \rightarrow \overline{\mcM}_{g, r}$" introduced in ~\cite[\S\,6.3]{Wak4} and the classifying morhism $S \rightarrow  \mcM_{g, r} \subseteq \overline{\mcM}_{g, r}$ of $\msX$.
 Here,  $\mcM_{g, r}$  denotes the moduli stack of $r$-pointed smooth proper curves of genus $g$ over $k$,  and  $\overline{\mcM}_{g, r}$ denotes its Deligne-Mumford compactification).
Each dormant $\mr{GL}_n^{(\NN)}$-oper induces
a dormant $\mr{PGL}_n^{(\NN)}$-oper via projectivization, and 
this assignment yields a morphism of $S$-stacks
\begin{align} \label{EQ195}
\mcO p^{^\mr{Zzz...}}_{\msX, n, \vec{a}, \mcL_\flat} \rightarrow \overline{\mcO} p^{^\mr{Zzz...}}_{\msX, n, \rho (\vec{a})}.
\end{align}

Next, 
for each $d \in \mbQ$,
  denote by $\mcP ic^d_{X/S}$ the relative Picard stack  of $X/S$  classifying  
line bundles on $X$ of relative degree $d$, where
$\mcP ic^d_{X/S} := \emptyset$ if $d \in \mbQ \setminus \mbZ$.
We set  $q := \mr{deg}(\mcL) + \frac{n (n-1)(2g-2+r)}{2}$.
Then, the assignment $\mcN \mapsto \mcN^{\otimes n}$ determines a morphism of $S$-stacks
\begin{align} \label{EQQ299}
[n] :  \mcP ic_{X/S}^{q/n} \rightarrow \mcP ic_{X/S}^{q}.
\end{align}
The fiber of this morphism over the $S$-rational point  classifying  $\mcL \otimes \Omega^{\otimes \frac{n (n-1)}{2}}$ defines a (possibly empty) $S$-stack
\begin{align}
\mcR_{n, \mcL}.
\end{align}
The structure morphism $\mcR_{n, \mcL} \rightarrow S$ of $\mcR_{n, \mcL}$ is quasi-finite, proper, and \'{e}tale (by the assumption $n < p$).
Also, it follows from \eqref{EQ456} that the assignment $(\msF^\heartsuit := (\mcF, \nabla, \{ \mcF^j \}_j), \eta) \mapsto \mcF^{n-1}$ gives an $S$-morphism
\begin{align} \label{EQ300}
\mcO p^{^\mr{Zzz...}}_{\msX, n, \vec{a}, \mcL_\flat} \rightarrow \mcR_{n, \mcL}.
\end{align}

\bpr\label{Prop49}
(Recall that we have assumed the inequality $n < p$.)
The following assertions hold.
\begin{itemize}
\item[(i)]
The morphism
\begin{align} \label{EQ322}
\mcO p^{^\mr{Zzz...}}_{\msX, n, \vec{a}, \mcL_\flat} \xrightarrow{\eqref{EQ195} \& \eqref{EQ300}}
\overline{\mcO} p^{^\mr{Zzz...}}_{\msX, n, \rho (\vec{a})}  \times \mcR_{n, \mcL}
\end{align}
is an isomorphism of $S$-stacks.
\item[(ii)]
The structure morphism  $\mcO p^{^\mr{Zzz...}}_{\msX, n, \vec{a}, \mcL_\flat} \rightarrow S$ of $\mcO p^{^\mr{Zzz...}}_{\msX, n, \vec{a}, \mcL_\flat}$ is (not representable  but) quasi-finite and  proper.
Moreover, if $n= 2$, then this morphism is faithfully flat.
\end{itemize}
\epr
\begin{proof}
To prove assertion (i), 
we  construct an inverse morphism of  \eqref{EQ322}.
The problem is  to construct a
functor
\begin{align} \label{EQ323}
\overline{\mcO} p^{^\mr{Zzz...}}_{\msX, n, \rho (\vec{a})} (S) \times \mcR_{n, \mcL} (S) \rightarrow \mcO p^{^\mr{Zzz...}}_{\msX, n, \vec{a}, \mcL_\flat} (S)
\end{align}
that is functorial  with respect to $S$ and defines an inverse to the fiber of  $\eqref{EQ322}$ over $S$.

Let us take a pair $(\msE^\spadesuit, \mcQ)$ consisting of a dormant $\mr{PGL}_n^{(\NN)}$-oper $\msE^\spadesuit$ on $\msX$ of radii $\rho (\vec{a})$ and a line bundle $\mcQ$ on $X$ classified by $\mcR_{n, \mcL}$.
By definition, $\mcQ$ admits an isomorphism  $\eta : \mcT^{\otimes \frac{n (n-1)}{2}} \otimes \mcQ^{\otimes n} \xrightarrow{\sim} \mcL$.
Then, $\nabla_\mcL$ corresponds, via this isomorphism, to a $\mcD^{(\NN -1)}$-module structure $\nabla_\vartheta$ on  $\mcT^{\otimes \frac{n (n-1)}{2}} \otimes \mcQ^{\otimes n}$.
The  pair $\vartheta := (\mcQ, \nabla_\vartheta)$ specifies a (dormant) $n^{(\NN)}$-theta characteristic of $\msX$, in the sense of ~\cite[\S\,5.4]{Wak4}.
By ~\cite[Theorem 5.5.1]{Wak4},
there exists a unique  dormant $(\mr{GL}_n^{(\NN)}, \vartheta)$-oper inducing $\msE^\spadesuit$ via projectivization; in particular, this defines a dormant $\mr{GL}^{(\NN)}_n$-oper $\msF^\heartsuit$, and $\eta$ forms an isomorphism between its determinant and $\mcL_\flat$.
Thus, the pair $(\msF^\heartsuit, \eta)$  specifies an object of   $\mcO p^{^\mr{Zzz...}}_{\msX, n, \vec{a}, \mcL_\flat} (S)$ (cf. the italicized comment in the second paragraph of this subsection).
The resulting assignment $(\msE^\spadesuit, \mcQ) \mapsto \msF^\heartsuit$  turns out to   define  the desired functor \eqref{EQ323}.
This completes the proof of assertion (i).

Moreover, assertion (ii) follows from ~\cite[Theorems B and C]{Wak4}.
\end{proof}

\bco \label{Co33}
(Recall that we have assumed the inequality $n < p$.)
Suppose 
the following two conditions:
\begin{itemize}
\item
The integer $\mr{deg}(\mcL) + \frac{n (n-1)(2g-2+r)}{2}$ is divided by $n$;
\item
$S = \mr{Spec}(k)$, i.e., $\msX$ is an $r$-pointed smooth proper curve over $k$ of genus $g$.
\end{itemize}
Then, the following equality holds:
\begin{align}
\sharp  (\mcO p^{^\mr{Zzz...}}_{\msX, n, \vec{a}, \mcL_\flat} (k)) 
= n^{2g} \cdot  \sharp (\overline{\mcO} p^{^\mr{Zzz...}}_{\msX, n, \rho (\vec{a})} (k) )< \infty,
\end{align}
where for a category $\mcC$  we denote by $\sharp (\mcC)$ the cardinality of the set of isomorphism classes of objects in $\mcC$.
\eco
\begin{proof}
Since the first assumption implies  $\sharp (\mcR_{n, \mcL}(k)) = n^{2g}$,
the assertion follows from Proposition \ref{Prop49}, (i).
\end{proof}

\begin{rem} \label{Rem78}
According to  an argument similar to ~\cite[Remark 8.21]{Wak2} together with Proposition \ref{Prop49}, (i) and (ii),
there exists  an open substack
\begin{align} \label{Eq891}
{^\circledcirc}\mcM_{g, r}
\end{align}
of $\mcM_{g, r}$ classifying pointed curves $\msX$ such that  $\mcO p^{^\mr{Zzz...}}_{\msX, n,  \vec{a}, \mcL_\flat}$  is  \'{e}tale whenever $\mcL_\flat$ is  a $p^\NN$-flat line bundle  whose exponent at $\sigma_i$ coincides with $\sum\limits_{j=1}^n a_i^{[j]}$ modulo $p^\NN$ for every $i$.
Since  $\mcM_{g, r}$ is irreducible, 
${^\circledcirc}\mcM_{g, r}$ is either empty or dense in $\mcM_{g, r}$.
In the case of  $n=2$, 
we can find at least one choice of $\vec{a}$ such that ${^\circledcirc}\mcM_{g, r}$ is dense (cf. ~\cite[Theorems B and  C]{Wak4}). 
\end{rem}

\subsection{Harder-Narasimhan polygons} \label{SS53}

We shall consider a characterization of $\mr{GL}_n^{(\NN)}$-opers by using their Harder-Narasimhan polygons; our result is a higher-level and parabolic  generalization of  ~\cite[Theorem 1.1.2]{JoPa}. 
In this subsection, 
we suppose that $S = \mr{Spec}(k)$, i.e., $\msX$ is an $r$-pointed smooth proper  curve of genus $g$ over $k$.

Recall that 
each  vector bundle $\mcF$ on $X$ admits the Harder-Narasimhan filtration
\begin{align}
0 = \mcF^m \subsetneq \mcF^{m-1} \subsetneq \cdots \subsetneq \mcF^1 \subsetneq \mcF^0 = \mcF.
\end{align}
That is to say, $\mcF^j  /\mcF^{j+1}$ are all semistable, and the slopes (in the usual sense) of these vector bundles satisfy
\begin{align}
\mu_{\mr{max}} (\mcF) := \mu  (\mcF^{m-1}/\mcF^m) >  \mu (\mcF^{m-2}/\mcF^{m-1}) > \cdots > \mu (\mcF^0/\mcF^1) =:\mu_{\mr{min}}(\mcF).
\end{align}
This filtration associates a polygon in the plane $\mbR^2$ whose vertices 
are given by the points  $(\mr{rk}(\mcF^j), \mr{deg}(\mcF^j))$ ($0 \leq j \leq m$).
We denote this polygon by 
\begin{align} \label{EQ510}
P_\mcF.
\end{align}
The segment of  $P_{\mcF}$ joining $(\mr{rk}(\mcF^j), \mr{deg}(\mcF^j))$ and $(\mr{rk}(\mcF^{j-1}), \mr{deg}(\mcF^{j-1}))$ has slope $\mu (\mcF^{j-1}/\mcF^j)$.

Let us consider 
the set $B$ of convex polygons in  $\mbR^2$ starting at $(0, 0)$ and ending at 
$(n, n a -  \frac{n (n-1) (2g-2+r)}{2})$; this set
 is partially ordered:
if $P_1$, $P_2$ are two such polygons, we say that $P_1 \succcurlyeq P_2$
if $P_1$ lies on or above $P_2$.
Also, this set contains a specific polygon
 with 
vertices $q_j := \left(j, ja - \frac{j (j-1)(2g-2+r)}{2}\right)$ for $0 \leq j \leq n$; we denote it by 
\begin{align}
P_{\mr{oper}}^a.
\end{align}
As asserted in the proposition below, this convex polygon characterizes $\mcD^{(\NN -1)}$-modules admitting a structure of  $\mr{GL}_n^{(\NN)}$-oper.

\bpr [cf. ~\cite{JoPa}, Theorem 1.1.2]\label{Eq332}
Let $(\N_1, \cdots, \N_r)$ be an $r$-tuple of positive integers and 
$\vec{a} := (a_1, \cdots, a_r)$  an element of $\prod_{i=1}^r \Xi_{\N_i, \NN}^{\leq}$ 
such that $s_1^{1} (\vec{a}) \in \prod_{i=1}^r \Xi_{\N_i, 1}^{\leq}$ and $s_2^{1} (\vec{a}) \in \prod_{i=1}^r \Xi_{\N_i, \NN -1}^{\leq}$ (cf. \eqref{EQ503}).
Also, let $n$ be a positive integer and $\msE := (\mcE, \vec{\mfe}, \vec{a}/p^\NN)$  a rank-$n$ parabolic bundle on $\msX^{(\NN)}$ of parabolic weights $\vec{a}/p^\NN$ whose parabolic Frobenius pull-back by $F^{(\NN -1)}_{X/S}$ is semistable.
Denote by $\msE^F := (\mcE^F, \nabla^F, \vec{\mfe}^F, \vec{a})$ the parabolic Frobenius pull-back of $\msE$ by $F^{(\NN)}_{X/S}$ and by 
\begin{align}
0 = \mcF^m \subsetneq \mcF^{m-1} \subsetneq \cdots \subsetneq \mcF^0 = \mcF
\end{align}
  the Harder-Narasimhan filtration of $\mcF := \mcE^F$.
Suppose that  $\mcF$ has  degree $n a - \frac{n (n-1)(2g-2+r)}{2}$ for  some  integer $a$.
Then, the following assertions hold.
\begin{itemize}
\item[(i)]
We have  $P_\mr{oper}^a \succcurlyeq P_{\mcF}$.
\item[(ii)]
The equality $P_\mr{oper}^a = P_{\mcF}$ holds if and only if the triple $(\mcF, \nabla^F,  \{ \mcF^j \}_j)$ forms  a (dormant) $\mr{GL}_n^{(\NN)}$-oper.
(If these conditions are  fulfilled, then $n = m$ and the equality $\mr{deg}(\mcF^{n-1}) = a$ holds.)
\end{itemize} 
\epr
\begin{proof}
First, we shall prove assertion (i).
Suppose that there exists a vertex of $P_\mcF$ that lies above $P_\mr{oper}^a$.
The segment $S_j$ joining the vertices $q_{j-1}$ and $q_{j}$ has slope $s_j := a -(j-1) (2g-2+r)$.
Since $s_j - s_{j+1} = 2g-2+r$,
it follows from an elementary argument about slopes of convex polygons that
there exists $j \in \{1, \cdots, m-1 \}$ satisfying $\mu (\mcF^j/\mcF^{j+1}) - \mu  (\mcF^{j-1}/\mcF^j) > 2g-2+r$.
This contradicts the result of Lemma \ref{Lem42} proved below.
Thus, we have  $P_\mr{oper}^a \succcurlyeq P_{\mcF}$, as desired.

Next, we shall prove assertion (ii).
Since the ``if" part of the required equivalence is clear, we only consider
the ``only if" part.
Suppose that   $P_\mr{oper}^a = P_{\mcF}$ (hence $m= n$).
As an intermediate step, 
we shall prove, by descending  induction on $j$, the claim that $\nabla^F_0 (\mcF^{j}) \subseteq \mcF^{j-1}$ ($j=1, \cdots, n$), where $\nabla^F_0$ denotes the log connection on $\mcF$ induced from $\nabla^F$.

The base step, i.e., the case of $j=n$, is clear.
To consider the induction step, we suppose that the claim with $j$ replaced by  $j+1$ has been proved.
The composite 
\begin{align} \label{EQ520}
\alpha : \mcF^{j} \xrightarrow{\mr{inclusion}} \mcF \xrightarrow{\nabla^F_0} \Omega \otimes \mcF \xrightarrow{\mr{quotient}} \Omega \otimes (\mcF/\mcF^{j})
\end{align}
 is verified to be $\mcO_X$-linear.
By  the induction assumption, $\alpha$ induces, via the quotient $\mcF^j \twoheadrightarrow \mcF^j/\mcF^{j+1}$,  an $\mcO_X$-linear morphism
 $\overline{\alpha} : \mcF^j/\mcF^{j+1} \rightarrow \Omega \otimes (\mcF/\mcF^j)$.
Since $\alpha$ (hence also $\overline{\alpha}$) is nonzero by Lemma \ref{Lem42}, 
$\mr{Im}(\overline{\alpha})$ turns out to be a rank one vector bundle.
Thus, 
there exists at least one integer $s < j$ such that the image of $\mr{Im}(\overline{\alpha})$ via  the natural quotient $\Omega \otimes (\mcF/\mcF^j) \twoheadrightarrow \Omega \otimes (\mcF/\mcF^{s+1})$ is nonzero;
let $s_\mr{min}$  be the minimum element among such  $s$'s. 
By definition,  $\overline{\alpha}$ induces an injection $\mcF^j/\mcF^{j+1} \hookrightarrow \Omega \otimes (\mcF^{s_\mr{min}}/\mcF^{s_\mr{min}+1})$.
But, 
the following equalities hold:
\begin{align}
&  \ \ \ \  \mr{deg}(\Omega \otimes (\mcF^{s_\mr{min}}/\mcF^{s_\mr{min}+1})) -\mr{deg}(\mcF^j / \mcF^{j+1}) \\
& = \left((2g-2+r) +a  - (n -s_\mr{min} -1)(2g-2+r) \right)
  -\left(a -(n-j-1)(2g-2+r) \right) \notag \\
& = (2g-2+r) (s_\mr{min}-j +1).
\end{align}
It follows that the equality $s_\mr{min} = j-1$ must be fulfilled, which means 
$\nabla^F_0 (\mcF^j) \subseteq \mcF^{j-1}$.
 This completes the proof of the induction step, and hence completes the proof of the claim.

By comparing the respective degrees, we see that 
the injection $\mcF^j/\mcF^{j+1} \xrightarrow{} \Omega \otimes (\mcF^{j-1}/\mcF^j)$ induced by $\nabla^F_0$ is bijective.
Therefore, the morphism \eqref{Eq59} for $(\mcF, \nabla^F, \{ \mcF^j \}_j)$ turns out to be an isomorphism.
That is to say, it forms a $\mr{GL}_n^{(\NN)}$-oper.
This completes the proof of the ``only if" part.
\end{proof}

The following lemma was used in the proof of the above proposition (cf. ~\cite[Lemma 5.1.1]{JoPa}, ~\cite[Lemma 4.2]{LasPa5}, ~\cite[Corollary 2]{She}, ~\cite[Theorem 3.1]{Sun5}).

\ble \label{Lem42}
Let us keep the notation and assumption  in  the above proposition.
Then, for each $j= 1, \cdots, m-1$,
the subbundle $\mcF^{j}$ of $\mcF$ is not closed under $\nabla^F$ and satisfies 
 the inequality  
\begin{align} \label{EQ466}
\mu (\mcF^{j}/\mcF^{j+1}) - \mu (\mcF^{j-1} / \mcF^{j}) \leq 2g-2 +r.
\end{align}
In particular, (since $m < \mr{rank}(\mcE)$)  we have 
\begin{align}
\mu_{\mr{max}} (\mcF) - \mu_{\mr{min}} (\mcF) \leq (\mr{rank}(\mcE) -1) (2g-2+r).
\end{align}
\ele
\begin{proof}
By the assumption on $\vec{a}$ together with Proposition \ref{Prop100},
we may assume that $\NN =1$.
Let us choose $j \in \{1, \cdots, m-1 \}$.
If  $\nabla_0^F$ is as in the proof of the above proposition,
then  the composite
\begin{align} \label{EQ501}
\mcF^{j} \xrightarrow{\mr{inclusion}}
\mcF \xrightarrow{\nabla^F_0} \Omega \otimes \mcF \xrightarrow{\mr{quotient}}
\Omega\otimes (\mcF/\mcF^{j})
\end{align} 
is verified to be $\mcO_X$-linear.

Suppose that $\mcF^j$ is closed under $\nabla^F_0$, i.e.,  \eqref{EQ501} becomes the zero map.
It follows from Proposition \ref{Prop26} that
the parabolic $p$-flat subbundle $\msE^{F, j}$ determined by $\mcF^j$ corresponds to a parabolic subbundle $\msG$ of $\msE$.
By \eqref{EQ500} and the definition of parabolic slope, we have 
\begin{align}
 \mr{par}\text{-}\mu (\msE)  &= \frac{1}{p} \cdot \left(
\mr{par}\text{-}\mu (\msE^F)\right)  \\
& = \frac{1}{p} \cdot \left( \mu (\mcF) - \frac{1}{n} \left(\sum_{i=1}^r \sum_{j'=1}^{\N_i} a_i^{[j']} \LL_i^{[j']}\right) \right) \notag \\
& <  \frac{1}{p} \cdot \left(\mu (\mcF^j) - \frac{1}{n} \cdot \left(\sum_{i=1}^r \sum_{j'=1}^{\M_i} a_i^{[j']} \LL_{i, \mcF^j}^{[j']} \right) \right)\notag \\
&= \frac{1}{p} \cdot \mr{par}\text{-}\mu (\msE^{F, j}) \notag \\
&=  \mr{par}\text{-}\mu (\msG)  \notag
\end{align}
where the ``$<$" follows from the fact that 
$\{ \mcF^j \}_j$ is the Harder-Narasimhan filtration (which  implies $\mu (\mcF^j) > \mu (\mcF)$).
This contradicts the semistability assumption on  $\msE$, and hence, proves the first assertion.

Moreover, this fact means that \eqref{EQ501} is nonzero, which 
implies
\begin{align}
\mu (\mcF^j/\mcF^{j+1}) = \mu_{\mr{min}} (\mcF^j) \leq \mu_{\mr{max}} (\Omega \otimes (\mcF/\mcF^j)) = \mu (\mcF^{j-1}/\mcF^j) + 2g-2+r.
\end{align}
This completes  the proof of the second assertion.
\end{proof}

\section{Maximally Frobenius-destabilized bundles} \label{S2}

Maximally Frobenius-destabilized bundles 
on unpointed curves
were  investigated in some references, e.g.,   ~\cite{JoPa}, ~\cite{JRXY}, ~\cite{JoXi},  ~\cite{LanPa},  and ~\cite{Zha}.
In particular, it has been shown that these bundles correspond bijectively to dormant opers via Cartier descent; this result enables us to resolve  related  counting problems  in the study of vector bundles in positive characteristic.
In this section, we describe the higher-level and parabolic generalization of this result  (cf. Theorem \ref{Prop3}).

We suppose that $S = \mr{Spec}(k)$, i.e., $\msX$ is an $r$-pointed smooth proper  curve of genus $g$ over $k$.

\LSP
\subsection{Maximally $F^{(\NN)}$-destabilized bundles} \label{SS28}

Let $n$ be a positive integer and 
  $\msE := (\mcE, \vec{\mfe}, \vec{a}/p^\NN)$ (where $\vec{\mfe} := (\mfe_1, \cdots \mfe_r)$)  
a rank-$n$ parabolic bundle on $\msX^{(\NN)}$ 
such that the type of the parabolic structure $\mfe_i$ at $\sigma_i $ coincides with   $1_{\times n} := (1, \cdots, 1) \in \mbZ_{>0}^{\times n}$ for every $i=1, \cdots, r$ and $\vec{a}$ belongs to $(\Xi_{n, \NN}^{\leq})^{\times r}$.
This bundle gives rise to  its parabolic Frobenius pull-back $\msE^F := (\mcE^F, \nabla^F, \vec{\mfe}^F, \vec{a})$ by $F^{(\NN)}_{X/S}$.
For convenience, we write $\mcF := \mcE^F$.

\bde \label{Def}
We shall say that $\msE$ is  {\bf maximally $F^{(\NN)}$-destabilized} 
if
there exists an $n$-step decreasing filtration  
\begin{align}
0 = \mcF^n \subsetneq \mcF^{n-1} \subsetneq \cdots \subsetneq \mcF^1 \subsetneq \mcF^0 = \mcF
\end{align}
such that (the subquotients $\mcF^j/\mcF^{j+1}$ are line bundles and)  the equality
\begin{align}
\left(\mu (\mcF^j/\mcF^{j+1}) - \mu (\mcF^{j-1}/\mcF^j) = \right) \mr{deg} (\mcF^j/\mcF^{j+1}) - \mr{deg} (\mcF^{j-1}/\mcF^j) = 2g-2+r
\end{align} 
holds  for every $j=1, \cdots, m-1$.
We shall refer to such a filtration as a {\bf destabilizing filtration} of $\msE$.
\ede

\bpr \label{Prop82}
Let $\msE$  be as above,
and let us consider 
the following three conditions:
\begin{itemize}
\item[(a)]
$\msE$ is maximally $F^{(\NN)}$-destabilized;
\item[(b)]
$s_1^1 (\vec{a}) \in (\Xi_{n, 1}^{\leq})^{\times r}$ and $s_2^1 (\vec{a}) \in (\Xi_{n, \NN -1}^{\leq})^{\times r}$ (cf. \eqref{EQ503});
\item[(c)]
The inequalities 
$\sum\limits_{i=1}^r (s_1^1 (a_i^{[n]})-s_1^1 (a_i^{[1]})) < \frac{n (2g-2+r)}{2} <  \frac{p}{n}$
 hold.
\end{itemize}
Then, the following assertions hold.
\begin{itemize}
\item[(i)]
If (a) is satisfied, then
a destabilizing filtration of $\msE$ is uniquely determined.
(Thus, it makes sense to speak of ``the" destabilizing filtration of $\msE$.)
\item[(ii)]
If $n \mid \mr{deg}(\mcE)$ and  all the conditions (a)-(c) are satisfied, then 
$\msE$ is stable.
\end{itemize}
 \epr
 \begin{proof}
  Assertion (i) can be proved by a routine argument using $2g-2+r > 0$ (so we omit the detail of the proof).

In the following,  we prove assertion (ii). 
  Suppose that there exists a proper parabolic subbundle $\msG := (\mcG, \vec{\mfg}, \vec{a}/p^\NN)$ of $\msE$ satisfying  $\mr{par}\text{-}\mu (\msG) \geq \mr{par}\text{-}\mu (\msE)$.
  By  Proposition \ref{Prop26},
  the parabolic Frobenius pull-back $\msG^F := (\mcG^F, \nabla^F_\mcG, \vec{\mfg}^F, \vec{a})$ specifies a parabolic $p^\NN$-flat subbundle of $\msE^F$.
  Denote by $m \left( < n \right)$ the rank of $\mcG^F$, and 
  define a decreasing filtration $\{ \mcG^{F, j} \}_{j=0}^n$ on $\mcG^F$ by putting $\mcG^{F, j} := \iota^{-1}(\mcF^j)$, where 
  $\{ \mcF^j \}_{j=0}^n$ denotes the destabilized filtration of $\msE$ and 
  $\iota$ denotes the natural inclusion $\mcG^F \hookrightarrow \mcF \left( =\mcF^0 =  \mcE^F\right)$.
  Then, one obtains a subset $B := \left\{ j \, | \, \mcG^{F, j} \neq \mcG^{F, j+1} \right\}$ of $\{0, \cdots, n-1 \}$.
 By Lemma \ref{Lem442} below together with 
  an argument similar to the proof of Proposition \ref{Eq332} (by using the assumption (b)),
  we see 
   that the log connection on $\mcF$ induced from $\nabla^F$ gives, in a usual manner,  an injection  $\mcF^j /\mcF^{j+1} \hookrightarrow \Omega \otimes (\mcF^{j-1}/\mcF^j)$ (i.e., the $j$-th Kodaira-Spencer map).
  By the condition (a), this morphism is verified to be an isomorphism; 
   it moreover restricts to an injection $\mcG^{F, j}/\mcG^{F, j+1} \hookrightarrow \Omega \otimes (\mcG^{F, j-1}/\mcG^{F, j})$.
It follows that $j \in B \cap \{ 1, \cdots, n-1\}$ implies $j-1  \in B$, and hence that $B = \{ 0, \cdots, m -1\}$.
Hence, the composite
$\mcG^F \xrightarrow{\mr{inclusion}} \mcF \twoheadrightarrow \mcF/\mcF^m$
  turns out to be injective.
  This implies
  \begin{align} \label{EQ522}
  \mu (\mcG^F) \leq \mu (\mcF/\mcF^m).
  \end{align}

On the other hand, if we write $c:= \mr{deg}(\mcF^{n-1})$, then we have
  \begin{align} \label{EQ524}
  \mr{par}\text{-}\mu (\msE) 
  &= \frac{1}{p^\NN \cdot n} \cdot \mr{deg}(\mcF) \\
  & = 
  \frac{1}{p^\NN \cdot n} \cdot \sum_{j=0}^{n-1}\mr{deg}(\mcF^j/\mcF^{j+1})  \notag \\
 & = \frac{1}{p^\NN\cdot n} \cdot  \sum_{j=0}^{n-1}\left(c-\frac{(n-j-1)(2g-2+r)}{2} \right) \notag \\
  & = 
   \frac{1}{p^\NN} \cdot \left( 
  c- \frac{(n-1)(2g-2+r)}{2}  \right), \notag
  \end{align}
  where  the first equality follows from \eqref{EQ500}
  and  the third equality follows from
  the assumption that $\msE$ is maximally $F^{(\NN)}$-destabilized.
  Thus,  by combining it with \eqref{EQ522}, we have
  \begin{align}
  \mr{par}\text{-}\mu (\msG) &=
  \frac{1}{p^\NN} \cdot \mu (\mcG^F) \\
&   \leq \frac{1}{p^\NN} \cdot \mu (\mcF/\mcF^m) \notag \\
&= \frac{1}{p^\NN \cdot m} \cdot \sum_{j=0}^{m-1} \mr{deg}(\mcF^j/\mcF^{j+1}) \\
& =  \frac{1}{p^\NN \cdot m} \cdot\sum_{j=0}^{m-1}  \left(c-\frac{(n-j-1)(2g-2+r)}{2} \right) \notag \\
& =  \frac{1}{p^\NN} \cdot \left( c-\frac{(2n - m-1)(2g-2+r)}{2} \right) \notag \\
& \stackrel{\eqref{EQ524}}{=} \mr{par}\text{-}\mu (\msE) + \frac{(m-n)(2g-2+r)}{p^\NN \cdot 2} \notag \\
&  < \mr{par}\text{-}\mu (\msE).
  \end{align}
  This is a contradiction, and we finish the proof of assertion (ii).
   \end{proof}

The following lemma was used in the proof of the above proposition.

\ble \label{Lem442}
Suppose that $n \mid \mr{deg} (\mcE)$ and  the conditions (a)-(c) in the above proposition are satisfied.
Then, the proper nontrivial subbundles $\mcF^j$ ($j=1, \cdots, n-1$) of $\mcF$ are  not closed under $\nabla^F$.
\ele
\begin{proof}
To prove the assertion, we are always free to replace $S$ with its \'{e}tale covering.
Hence, since $n \mid \mr{deg} (\mcE)$, we may assume, after possibly tensoring  $\mcE$ with a line bundle, that $\mr{deg}(\mcE) = 0$.
By the condition (b) together with Proposition \ref{Prop100},
we may assume that $\NN =1$. 

Now, suppose that $\mcF^j$ (for some $j=1, \cdots, n-1$) of $\mcF$ is  closed under $\nabla^F$.
If  $\nabla^j$ denotes the log connection on $\mcF^j$ obtained by restricting $\nabla^F$, then we have $F^*_{X/S}\mcS ol (\nabla^j) = F^*_{X/S} \mcS ol (\nabla) \cap \mcF^j$ (where $F_{X/S} := F^{(1)}_{X/S}$), and  both sides of this equality have degree divided by $p$.
Let us observe the following equalities:
\begin{align}
\mr{deg}(\mcF) &= \sum_{j'=0}^{n-1} \mr{deg}(\mcF^{j'}/\mcF^{j'+1}) \\
& = \sum_{j' = 0}^{n-1} \left(\mr{deg}(\mcF^{n-1}) -(n- j' -1)(2g-2+r) \right)\\
& =  n \cdot \mr{deg}(\mcF^{n-1}) -\frac{n (n-1)(2g-2+r)}{2}. \notag
\end{align}
Hence,  the equality $\mr{deg}(\mcE) = 0$ and 
\eqref{EQ21} together imply 
\begin{align}
\sum_{i=1}^r \sum_{j=1}^n a_{i}^{[j]} \left(= p \cdot \mr{deg}(\mcE) + 
\sum_{i=1}^r \sum_{j=1}^n a_{i}^{[j]} \right)=n \cdot \mr{deg}(\mcF^{n-1})  -\frac{n (n-1)(2g-2+r)}{2}.
\end{align}
It follows that
\begin{align}
\mr{deg}(\mcF^j) &= \sum_{j' = j}^{n-1} \mr{deg}(\mcF^{j'}/\mcF^{j' +1}) \\
& = \sum_{j' = j}^{n-1} \left(\mr{deg}(\mcF^{n-1})  - (n-j' -1)(2g-2+r) \right) \notag \\
& = (n-j)\cdot \mr{deg}(\mcF^{n-1})  -\frac{(n-j)(n-j-1) (2g-2+r)}{2} \notag \\
& =  \frac{j (n-j)(2g-2+r)}{2} + \frac{n-j}{n} \cdot   \sum_{i=1}^r \sum_{j=1}^{n} a_i^{[j]}.
\end{align}

Next, observe that 
\begin{align}
& \ \ \ \ \mr{deg}(F^*_{X/S}\mcS ol (\nabla)\cap \mcF^j)  \\
& \leq \mr{deg}(\mcF^j)  - \sum_{i=1}^r \sum_{j' =1}^{n-j}a_{i}^{[j']}
 \notag  \\
& =  \frac{j (n-j)(2g-2+r)}{2} + \frac{1}{n} \cdot \sum_{i=1}^r  \left((n-j) \cdot \sum_{j'=n-j+1}^{n} a_i^{[j]} - j \cdot  \sum_{j' =1}^{n-j} a_i^{[j']}  \right)
  \notag \\
& = \frac{1}{n} \cdot \sum_{j'= n-j+1}^n \sum_{j''=1}^{n-j} \left(\frac{n (2g-2+r)}{2} + \sum_{i=1}^r (a_i^{[j']} - a_i^{[j'']})\right) \notag \\
& < \frac{1}{n} \cdot \frac{n^2}{2} \cdot \left(\frac{n (2g-2+r)}{2} + \frac{n (2g-2+r)}{2} \right) \notag \\
& \leq p,
\end{align}
where the inequality ``$<$" follows from the first inequality in the condition (c) and the last inequality ``$\leq$" follows from the second one in (c).
On the other hand, we have
\begin{align}
& \ \ \ \ \mr{deg}(F^*_{X/S}\mcS ol (\nabla)\cap \mcF^j)  \\
& \geq \mr{deg}(\mcF^j) - \sum_{i=1}^r \sum_{j'=j+1}^n a_i^{[j']} \notag  \\
& =   \frac{j (n-j)(2g-2+r)}{2} +
\frac{1}{n} \cdot \sum_{i=1}^r \left( (n-j) \cdot \sum_{j' = 1}^j a_i^{[j']} -j \cdot \sum_{j' = j+1}^n a_i^{[j']}\right)
  \notag\\
& =   \frac{1}{n} \cdot \sum_{j'= 1}^j \sum_{j''=j+1}^{n} \left(\frac{n (2g-2+r)}{2} + \sum_{i=1}^r (a_i^{[j']} - a_i^{[j'']})\right)
 \notag \\
& > 0.
\end{align}
These computations contradict the fact that  $\mr{deg}(F^*_{X/S}\mcS ol (\nabla)\cap \mcF^j)$ is divided by $p$.
This completes the proof of this lemma.
\end{proof}

Let $\mcL$ be a line bundle on $X^{(\NN)}$ with $n \mid \mr{deg}(\mcL)$ and $\vec{a}$ an element  of $(\Xi_{n, \NN}^{<})^{\times r} \left(\subseteq (\Xi_{n, \NN}^{\leq})^{\times r}\right)$ as before such that 
the conditions (b) and (c) described in Proposition \ref{Prop82} are fulfilled.
 Also, we set $\vec{1}_{\times n} := (1_{\times n}, \cdots, 1_{\times n})$ (= the $r$-tuple of copies of $1_{\times n}$) and $\omega := (n, \vec{1}_{\times n}, \vec{a}/p^\NN)$.
 Then, it follows from Proposition \ref{Prop82}, (ii), that
 there exists a fibered (full)  subcategory
\begin{align} \label{EQ178}
\mcU^F_{\msX^{(\NN)},  \omega, \mcL}
\end{align}
 of  $\mcU_{\msX^{(\NN)}, \omega,  \mcL}$
 classifying
maximally $F^{(\NN)}$-destablized stable parabolic bundles.

\subsection{Correspondence with dormant $\mr{GL}_n^{(\NN)}$-opers} \label{SS28}

 We keep the notation and assumptions introduced at the end of  the previous subsection, and 
  write
$\mcL^F_\flat := ((F^{(\NN)*}_{X/S}\mcL)(D_\star), \nabla^\mr{can}_{\mcL, D_\star})$ (cf. \eqref{Eq214}),
where $D_\star := \sum\limits_{i=1}^r ( \sum\limits_{j=1}^n a_i^{[j]}) \sigma_i$.

\bt [cf. Theorem \ref{ThA}] \label{Prop3}
\begin{itemize}
\item[(i)]
Let $\msE := (\mcE, \vec{\mfe}, \vec{a}/p^\NN)$ be the parabolic 
bundle classified by an $S$-rational point of $\mcU^F_{\msX^{(\NN)}, \omega, \mcL}$.
 Denote by $\{ \msE^{F, j} \}_{j=0}^n$ be the destabilizing filtration of $\msE^F$.
Then, the collection of data 
\begin{align} \label{Eq250}
\msE^{F, \heartsuit} := (\mcE^F, \nabla^F, \{ \mcE^{F, j}\}_j)
\end{align}
 forms a dormant $\mr{GL}_n^{(\NN)}$-oper 
 classified by $\mcO p^{^\mr{Zzz...}}_{\msX, n, \vec{\alpha}, \mcL^F_\flat}$ (cf. \eqref{Eq405}).
\item[(ii)]
Suppose that $\sum\limits_{j=1}^{n} a_i^{[j]}<p^\NN$ for every $i=1, \cdots, r$.
Then, the assignment $\msE \mapsto \msE^{F, \heartsuit}$ defines an isomorphism of fibered categories 
\begin{align}
\mcU^F_{\msX^{(\NN)}, \omega, \mcL} \xrightarrow{\sim} \mcO p_{\msX, n, \vec{a}, \mcL^F_\flat}^{^\mr{Zzz...}}.
\end{align}
\item[(iii)]
Let us keep the assumption in (ii).
Suppose further  that $n$ does not divide the integer  $\frac{n (n-1)(2g-2+r)}{2} + \sum\limits_{i=1}^r \sum\limits_{j=1}^n a_i^{[j]}$.
Then, 
$\mcU^F_{\msX^{(\NN)}, \omega, \mcL}$ is empty.
\end{itemize}
\et
\begin{proof}
First, we shall consider assertion (i).
By Proposition  \ref{Prop82}, (ii), the parabolic Frobenius pull-back of $\msE$ by  $F_{X/S}^{(\NN -1)}$ is stable.
Hence, the assertion  follows from Proposition \ref{Eq332}, (ii).

Next, we shall consider assertion (ii).
The assignment $\msE \mapsto \msE^{F, \heartsuit}$ is functorial because of its construction, so it defines a functor
\begin{align} \label{EQ544}
\mcU^F_{\msX^{(\NN)}, \omega, \mcL} \rightarrow  \mcO p_{\msX, n, \vec{a}, \mcL^F_\flat}^{^\mr{Zzz...}}.
\end{align}
We shall construct its inverse.
Let $(\msF^\heartsuit, \eta)$ (where $\msF^\heartsuit := (\mcF, \nabla, \{ \mcF^j \}_{j=0}^n)$) be an object classified by $\mcO p^{^\mr{Zzz...}}_{\msX, n, \vec{a}, \mcL_\flat^F}$.
Apply Lemma \ref{EQ244} described below to the restriction of $(\mcF, \nabla)$ to the formal neighborhood of each $\sigma_i$.
Then,  there exists a quasi-parabolic structure $\vec{\mff}$ on $\mcF$, for which $(\mcF, \nabla, \vec{\mff}, \vec{a})$ specifies a parabolic $p^\NN$-flat bundle.
It follows from Theorem \ref{Prop11}  that the horizontal  parabolic bundle
 $\msF^\nabla := (\mcF^\nabla, \vec{\mff}^\nabla, \vec{a}/p^\NN)$ 
 associated to this collection 
  is maximally $F^{(\NN)}$-destabilized.
  Moreover, by the assumptions  on $a_i^{[j]}$'s imposed in (ii), 
   we see that $\msF^\nabla$ belongs to
  $\mcU^F_{\msX^{(\NN)}, \omega, \mcL}(S)$ (cf. Proposition \ref{Prop82}, (ii)).
The formation of the assignment $(\msF^\heartsuit, \eta) \mapsto \msF^\nabla$ is functorial with respect to $S$, and hence, determines a functor 
\begin{align}
\mcO p_{\msX, n, \vec{a}, \mcL^F_\flat}^{^\mr{Zzz...}} \rightarrow \mcU^F_{\msX^{(\NN)}, \omega, \mcL}.
\end{align}
By  Theorem \ref{Prop11}, this functor defines
an inverse to \eqref{EQ544}.
So we finish the proof of assertion (ii).

Finally,
since $\frac{n (n-1)(2g-2+r)}{2} + \sum\limits_{i=1}^r \sum\limits_{j=1}^n a_i^{[j]}$  equals the degree of $(F_{X/S}^{(\NN)*}\mcL)(D_\star)$ modulo $n$, 
 assertion (iii) follows from assertion (ii) and the fact mentioned at the end of Remark \ref{EQ201}  
\end{proof}

The following lemma was used in the proof of the above theorem.

\ble \label{EQ244}
Let $(\mcF, \nabla)$ be a
rank-$n$ $p^\NN$-flat bundle on $\msU$ whose exponent is given by
 $[a_\circ^{[1]}, \cdots,  a_\circ^{[n]}]$ for some  element
$a_\circ := (a_\circ^{[1]}, \cdots, a_\circ^{[n]})$ of $\Xi_{n, \NN}^{<}$.
Then, there exists a unique quasi-parabolic structure $\mff_\circ$ on $\mcF$ for which $(\mcF, \nabla, \mff_\circ)$ forms a parabolic $p^\NN$-flat bundles.
\ele
\begin{proof}
Denote by $d^{[j]}$ ($j=1, \cdots, n$) the image of $a_{\circ}^{[j]}$ via the quotient $\mbZ \twoheadrightarrow \mbZ/p^\NN \mbZ$.
To prove this assertion, we may assume, without loss of generality, that
there exists an isomorphism of $\mcD_\circ^{(\NN -1)}$-modules  $\xi :  (\mcF, \nabla) \xrightarrow{\sim} \bigoplus_{j=1}^n \msO_{d^{[j]}, \flat}$. 
By considering the eigenspaces of the monodromy operator of $\bigoplus_{j=1}^n \nabla_{d^{[j]}}$ at $\sigma_\circ$,
it is verified that there exists a unique quasi-parabolic structure $\mff'_\circ$ on $\mcO_U^{\oplus n}$ for which the collection  $(\mcO_U^{\oplus n}, \bigoplus_{j=1}^n \nabla_{d^{[j]}}, \mff'_\circ)$ forms a parabolic $p^\NN$-flat bundle.
The quasi-parabolic structure $\mff_\circ$ on $\mcF$ corresponding to $\mff'_\circ$  via $\xi$ provides  a required one.

Moreover, since $a_\circ^{[1]} < \cdots < a_\circ^{[n]}$,  it follows from the definition of $\mff'_\circ$ and  ~\cite[Proposition 4.3.4, (ii)]{Wak4} that $\mff'_\circ$ is invariant under transposing by  any automorphism of $\bigoplus_{i=1}^n \msO_{d_\circ^{[j]}, \flat}$.
Hence, $\mff_\circ$ does not depend on the choice of $\xi$.
This proves the uniqueness assertion, as desired.
\end{proof}

By combining Proposition \ref{Prop49} and  Theorem \ref{Prop3}, (ii) and (iii),
 we obtain  the following assertion.

\bco \label{Cor1}
Let us keep the above notation and assumptions, and suppose that  $\sum\limits_{j=1}^n a_i^{[j]} < p^\NN$ for every $i=1, \cdots, r$.
Then, the following assertions hold.
\begin{itemize}
\item[(i)]
 $\mcU^F_{\msX^{(\NN)}, \omega, \mcL}$ may be represented by a (possibly empty) stack over $S$ whose structure morphism 
  $\mcU^F_{\msX^{(\NN)}, \omega, \mcL} \rightarrow S$ is (not representable but) quasi-finite and proper.
Moreover, 
if $n=2$, then this morphism is faithfully flat.
\item[(ii)]
$\mcU^F_{\msX^{(\NN)}, \omega, \mcL}$ is nonempty if and only if 
$\overline{\mcO} p^{^\mr{Zzz...}}_{\msX, n, \rho(\vec{a})}$ is nonempty and  $n$ divides the integer $\frac{n(n-1)(2g-2+r)}{2} +\sum\limits_{i=1}^r \sum\limits_{j=1}^n a_i^{[j]}$. 
\end{itemize}
\eco

Next, we perform an explicit computation of the number of maximally $F^{(1)}$-destabilized parabolic bundles of rank $2$.
In the non-parabolic case, this computation  was carried out  in ~\cite{Wak1} and ~\cite{Wak2} by applying  a correspondence with dormant $\mr{PGL}_2$-opers.

\bt [cf. Theorem \ref{ThB}] \label{Th22}
Let us take a line bundle $\mcL$ on $X^{(1)}$ of even degree  and  an element  $\vec{a} := (a_1, \cdots, a_r) \in (\Xi_{2, 1}^{<})^{\times r}$, where $a_i := (a_i^{[1]}, a_i^{[2]})$ ($i=1, \cdots, r$).
We write $\omega := (2, \vec{1}_{\times 2}, \vec{a}/p)$, and 
suppose that 
$r + \sum\limits_{i=1}^r (a_i^{[1]} + a_i^{[2]})$ is even and $\sum\limits_{i=1}^r (a_i^{[2]}-a_i^{[1]}) <2g-2+r <  \frac{p}{2}$.
Then, the  cardinality $\sharp (\mcU^F_{\msX^{(1)}, \omega, \mcL} (k))$ of the set of isomorphism classes of objects in $\mcU^F_{\msX^{(1)}, \omega, \mcL} (k)$ satisfies the following inequality:
 \begin{align}
  \sharp (\mcU^F_{\msX^{(1)}, \omega, \mcL} (k)) 
  \leq   2 \cdot p^{g-1} \cdot  \sum_{j =1}^{p-1} \frac{\prod\limits_{i=1}^r (-1)^{(j+1)(a_i^{[2]}-a_i^{[1]}+1)} \sin \left( \frac{(a_i^{[2]}-a_i^{[1]}) j \pi}{p}\right)}{\sin^{2g-2+r} \left(\frac{j \pi}{p} \right)}.
  \end{align}
 Moreover, this inequality becomes an equality when $\msX$ is sufficiently general,  or more precisely, classified by a point of ${^\circledcirc}\mcM_{g, r}$ (cf. \eqref{Eq891}) for $(n, \NN)=(2, 1)$.
\et
\begin{proof}
By  Corollaries \ref{Co33},  \ref{Cor1}, and ~\cite[Theorem 7.41]{Wak2},
we have 
 \begin{align} \label{EQ570}
  \sharp (\mcU^F_{\msX^{(1)}, \omega, \mcL} (k)) 
  \leq   2 \cdot p^{g-1} \cdot  \sum_{j =1}^{p-1} \frac{\prod_{i=1}^r \sin \left( \frac{(2 \cdot \tau (a_i^{[2]}-a_i^{[1]}) +1) j \pi}{p}\right)}{\sin^{2g-2+r} \left(\frac{j \pi}{p} \right)},
 \end{align}
where, for an integer $b$ with $0 \leq b \leq p-1$,
we set  $\tau (b) := \frac{b-1}{2}$ (resp., $\tau (b):= \frac{p-1-b}{2}$)
 if $b$ is odd (resp., even).
 On the other hand,  the following equality holds: 
 \begin{align} \label{EQ571}
 \sin \left( \frac{(2 \cdot \tau (a_i^{[2]}-a_i^{[1]}) +1) j \pi}{p}\right) = (-1)^{(j+1)(a_i^{[2]}-a_i^{[1]}+1)}\cdot \sin \left(\frac{(a_i^{[2]}-a_i^{[1]} ) j\pi}{p} \right). 
 \end{align}
Thus,  the first assertion  follows from \eqref{EQ570} and \eqref{EQ571}.

The second assertion follows from the generic \'{e}taleness of $\overline{\mcO}p^{^\mr{Zzz...}}_{\msX, 2, \rho (\vec{a})}/S$ proved in ~\cite[Chap.\,II, Theorem 2.8]{Mzk2}.
\end{proof}

\subsection*{Acknowledgements} 
We are grateful for the many constructive conversations we had with {\it algebraic curves in positive characteristic}, who live in the world of mathematics!
Our work was partially supported by Grant-in-Aid for Scientific Research (KAKENHI No. 21K13770).


\end{document}